\documentclass[a4paper, 11pt, reqno]{amsart}

\usepackage{setspace}\usepackage{mathrsfs}
\usepackage[utf8]{inputenc}
\usepackage{graphicx}\usepackage{constants}
\usepackage{amsmath}
\usepackage{amssymb,amsthm}
\usepackage{cases}
\usepackage{esint}
\usepackage[a4paper,margin=3cm]{geometry}
\usepackage[citecolor=blue, linkcolor=blue,colorlinks]{hyperref}

\graphicspath{ {images/} }
\newtheorem{theorem}{Theorem}
\newtheorem{lemma}{Lemma}
\newtheorem{cor}{Corollary}
\newtheorem{prop}{Proposition}

\theoremstyle{definition}
\newtheorem{rk}{Remark}

\DeclareMathOperator*{\supp}{supp}

\def\R{\mathbb{R}}

\def\half{{\textstyle\frac{1}{2}}}
\def\z{\zeta}
\def\a{\alpha}
\def\e{\epsilon} 

\def\p{\partial}

\font\smathbold=msbm7\font\ssmathbold=msbm7 at 5.5pt
\def\sR{{\hbox{\smathbold\char82}}}\def\ssR{{\hbox{\ssmathbold\char82}}}
\def\Br{B_{\sR^n}}
\def\ni{\noindent}

\def\ba{\begin{aligned}}
\def\ea{\end{aligned}}

\def\wt{\widetilde}
\def\b{\beta}

\def\be{\begin{equation}}\def\ee{\end{equation}}\def\bc{\begin{cases}}\def\ec{\end{cases}}
\def\inj{{\rm {inj}}}\def\Z{{\mathbb Z}}\def\N{{\mathbb N}}
\def\qed{\rightline{\setlength{\fboxsep}{0pt}\setlength{\fboxrule}{0.2pt}\fbox{\rule[0pt]{0pt}{1.3ex}\rule[0pt]{1.3ex}{0pt}}}}

\def\ric{{\rm{Ric } }}

\newcommand\fH[1]{\sbox0{#1}\dimen0=\ht0 \advance\dimen0 -1ex
  \sbox2{\'{}}\sbox2{\raise\dimen0\box2}%
  {\ooalign{\hidewidth\kern.1em\copy2\kern-.5\wd2\box2\hidewidth\cr\box0\crcr}}}

\font\twelvemi=cmmi12 at 12pt\font\elevenmi=cmmi11 at 9 pt
\renewcommand{\chi}{\raisebox{.13\baselineskip}{\hbox{\twelvemi\char31}}}
\newcommand{\gam}{{\raisebox{.08\baselineskip}{\hbox{\twelvemi\char13}}}}
\newcommand{\sgam}{{\raisebox{.0\baselineskip}{\hbox{\elevenmi\char13}}}}

\renewcommand{\gamma}{\gam}
\def\vol{{\rm {Vol }}}

\vskip 0.1in\def\grn_#1{{G_#1^{\sR^{\hskip-.007in n}}}}\def\sgrn_#1{{G_#1^{\ssR^{\hskip-.007in n}}}}
\def\per{{\rm {Per }}}

\def\nn{{\frac{n-1}{n}}}\def\nnn{{\frac n{n-1}}}
\def\PS{P\'olya-Szeg\H o }\def\gp{{g_{p}^{}}}
\def\vt{\tilde v}\def\gpt{{\wt g_p}}\def\rh{r_{\hskip-.2em\scriptscriptstyle H} }\def\diam{{\rm diam }}
\def\im{I_{\hskip-.1em\scriptscriptstyle M} }
\def\dh{\delta_{\hskip-.1em\scriptscriptstyle H} }
\title[Sharp Sobolev inequality] {Sharp Sobolev inequalities on noncompact Riemannian manifolds with bounded Ricci curvature}
\author[ C. Morpurgo, S. Nardulli, L. Qin] { Carlo Morpurgo,  Stefano Nardulli, Liuyu Qin}
 \thanks{The third author was supported by the National Natural Science Foundation of China (12201197)}
\begin{document}

\begin{abstract} 
Given a smooth, complete Riemannian manifold $M$ with bounded Ricci curvature and positive injectivity radius, we derive a sharp Sobolev inequality for the embedding of $W^{1,p}(M)$ into $L^{\frac{np}{n-p}}(M)$, when $1\le p< n$. We will first reduce the inequality to functions having support with small enough volume. In turn, we will show that the inequality for small volumes is implied by a first order uniform asymptotic expansion of the isoperimetric profile for $M$, for small volumes. We will then show that such an expansion follows from a local, uniform Sobolev inequality for functions in  $W^{1,1}$, having support with small enough diameter.

\end{abstract}
\numberwithin{equation}{section}
\maketitle
\tableofcontents

\section{Introduction}\label{1}

Let $(M,g)$ be a  smooth, connected, complete, noncompact,  $n-$dimensional  Riemannian manifold without boundary, $n\ge2$, and let $\ric, \inj(M)$ denote the Ricci curvature and the injectivity radius, respectively. The main result of this paper is the following:

\medskip
\begin{theorem}\label{thm} On a smooth, complete, connected, noncompact,  $n-$dimensional  Riemannian manifold without boundary $(M,g)$, suppose that 
\be |\ric |\le K,\qquad \inj(M)>0,\label{ric}\ee
for some $K\ge0.$
 Then, there exists $B$ such that for any $ p\in[1,n)$ 
\be \|u\|_q\le K(n,p)\|\nabla u\|_p+B\|u\|_p,\qquad q=\frac{np}{n-p},\quad u\in W^{1,p}(M),\label{main}\ee
where $K(n,p)$ denotes the best  constant in the Sobolev embedding $W^{1,p}(\R^n)\hookrightarrow L^q(\R^n)$.
\end{theorem}

When $M=\R^n$ the inequality \eqref{main} is valid with $B=0$ and for an explicit value of $K(n,p)$ - this is the well known result obtained by Aubin \cite{aubin} and  Talenti \cite{talenti}. The case $p=1$, which was obtained first by Federer-Fleming  \cite{fed-flem} and Fleming-Rishel \cite{flem-rish}, is equivalent to the classical isoperimetric inequality  on $\R^n$:
\be |\Omega|^\nn\le K(n,1)\per(\Omega),\qquad  K(n,1)=\frac1 n B_n^{-\frac1n},\ee
where $B_n$ denotes the volume of the unit ball on $\R^n$.

On general complete manifolds, inequality \eqref{main} is known for $1<p<n$ under stronger assumptions on the curvature tensor $R$, namely
\be|R|\le K,\quad  |\nabla R|\le K, \quad\inj(M)>0,\label{R}\ee
 by results due to Aubin-Li \cite{aubin-li}.  This includes the case of compact manifolds, which was settled independently by Druet \cite{druet-mathann}. The inequality for  $p=1$ is only known in certain cases: when $M$ compact  \cite{druet-gd},  when $n=2$ and $M$ has bounded sectional curvature, or $n\ge 3$ and $M$ has constant sectional curvature \cite{aubin-jdg}.

We point out that in the aforementioned papers the authors actually dealt  with stronger forms of the  Sobolev inequality. Namely, they considered inequalities of type
\be  \|u\|_q^a\le K(n,p)^a\|\nabla u\|_p^a+B\|u\|_p^a,\qquad u\in W^{1,p}(M),\label{Ipa}\tag{$I_p^a$}\ee
for $a>0$ and with $q$ as in \eqref{main}. The notation $I_p^a$ is borrowed from Hebey \cite{hebey-courant}, who denoted such inequalities as $I_{p,opt}^a$. First, let us mention that the constant $K(n,p)^a$ in \eqref{Ipa} is optimal, that is, it cannot be replaced by a smaller constant. This is a general fact, valid on any Riemannian manifold, coming from the fact that a  Riemannian metric is locally Euclidean (see Prop. 4.2 in \cite{hebey-courant} for a proof when $a=1$, which also works for any $a>0$). Secondly,   it's easy to see that if $a>b$ and $I_p^a$ holds then $I_p^b$ also holds. In \cite{aubin-li} the authors settled a conjecture due to Aubin for $1<p<n$, by proving that under \eqref{R} $I_p^p$ holds for $1<p\le2$ and $I_p^2$ holds for $2<p<n$. The same result was proved independently by Druet \cite{druet-mathann} on compact manifolds. In \cite{aubin-li}  the authors also deal with $I_p^a$ when $2<p<n$, and $n\ge 3$,  and prove it true (under \eqref{R}) for certain ranges of $a$ and $p$ (improving $I_p^2$).  In another paper \cite{druet-jfa}, Druet  also proved that inequality $I_p^p$ is false if $4<p^2<n$ and the scalar curvature is positive somewhere. 

We note that $I_2^2$ for $n\ge 3$ was first proved by Hebey-Vaugon \cite{hebey-vaugon-duke} under \eqref{R}, and also under \eqref{ric} and $M$ conformally flat outside a compact set.

Inequality \eqref{main} belongs to the so-called ``A part'' of the famous AB-program  initiated by E. Hebey, concerning  optimal values of $A, B$   in the family of inequalities 
\be \|u\|_q^a\le A\|\nabla u\|_p^a+B\|u\|_p^a,\qquad 1\le p<n,\quad q=\frac{np}{n-p},\quad u\in W^{1,p}(M)\label{AB}\ee
with particular attention to the cases $a=1, a=2, a=p$. See \cite{hebey-courant} and \cite{druet-hebey} for a detailed exposition of the program, and of the  results obtained up to the late 1990s-early 2000s. More recently, the AB program has been reloaded in the context of compact CD and RCD spaces, see for example \cite{nobili-violo}, \cite{nobili}, where  rigidity and stability issues  of optimal Sobolev  inequalities are considered.
\def\avr{{\rm {AVR}}}

Very recently, several authors have been studying sharp Sobolev  inequalities on noncompact manifolds with $\ric\ge0 $ and with  Euclidean volume growth (EVG). In \cite{bk} the authors prove that on such manifolds
\be \|u\|_q\le  \avr^{-\frac 1n}K(n,p)\|\nabla u\|_p,\qquad u\in W^{1,p}(M)\label{sobavr}\ee
where $\avr\in(0,1]$ is the so-called Asymptotic Volume Ratio (arising from the Bishop comparison theorem), and where the constant is sharp. This inequality has been extended in the context of metric measure spaces in \cite{nobili-violo}. Unless $M$ is isometric to $\R^n$ (which means $\avr=1$) one cannot have $B=0$ in \eqref{AB}
when $A=K(n,p)$, even assuming only $\ric\ge 0$ (see \cite{ledoux}, \cite{hebey-courant} Theorem 8.8).  Such ``Euclidean type inequalities'' (see \cite{hebey-courant} Section 8) in general   hold only  in special cases, other than $\R^n$,  for example on Cartan-Hadamard manifolds of dimensions 2,3,4 (\cite{hebey-courant} Theorems 8.5, 8.6, 8.7).
On the other hand, we know that if $M$ has EVG, $\ric\ge0$ and \eqref{R} holds (or just \eqref{ric} holds, in view of our new result), then \eqref{AB} holds with $A=K(n,p)$, for some $B$. What this is saying,  is that the best constant $A$ gets worse if we replace the full Sobolev norm with the seminorm  $\|\nabla u\|_p$. This phenomenon, on spaces with EVG,  was already observed in \cite{fmq}, at the level of Moser-Trudinger inequalities, i.e. Sobolev inequalities  for the borderline space $W^{\alpha,n/\alpha}(M)$ (see \cite{fmq} Thm. 1, Thm. 2). The main reason for such behavior is that the sharp Green functions asymptotics for large distances behave like those for small distances, but with a worse constant involving AVR.

\smallskip
Returning to \eqref{Ipa}, the proof of the known results on noncompact manifolds, as far as we can tell, are all based 
on two main (interchangeable) steps. The first step is to  reduce the global problem to a local one, that is to an inequality like \eqref{Ipa} valid for functions whose supports are small enough in diameter. This is typically done via partition of unity or compact exhaustion, see for example \cite{hebey-vaugon-duke}, Lemma 5, \cite{aubin-li} proofs of Thms. 1.1, 1.2, 1.3.

The second step, and the most difficult one, is to establish the local inequality. This is done via PDE methods, typically by a deep blowup analysis of solutions of the  Euler-Lagrange equation associated to the functional  
\be \frac{\|\nabla u\|_p^a+B\|u\|_p^a}{\|u\|_q^a}.\ee
The bounded geometry assumptions in \eqref{R} were essential in the arguments, since they guarantee a uniform, second order approximation of the metric in terms of the Euclidean metric, in geodesic polar coordinates. In \cite{hebey-vaugon-duke} the authors were able to avoid  \eqref{R} assuming only \eqref{ric},  but at the expense of assuming $(M,g)$ conformally flat at infinity, thereby reducing the proof of the local inequality to compact sets,  or to simply connected sets where $g$ is conformally flat.

In this paper we will follow a different approach to globalization, namely we will first prove that the global inequality \eqref{main} can be deduced from  the validity of the same inequality for functions whose supports have small {\emph {measure}}, as opposed to small diameter. This idea came from a similar result obtained recently in \cite{fmq-ann}, Thm 4, related to Moser-Trudinger inequalities. We observe in this regard  that in \cite{fmq-ann}  the set $E$ in the definition of (LMT) was assumed to be bounded, but this is not necessary. The concept of ``local'' in \cite{fmq-ann} is intended to be in terms of the measure, not the metric.

The second idea of this paper, which we believe is new, is that to prove \eqref{main} for small volumes it's enough to have a uniform asymptotic control of the isoperimetric profile of $M$ for small volumes. Specifically, we will show that if $\im(v)=\inf\{\per(E),\, \vol(E)=v\}$ denotes the isoperimetric profile of $M,$ then the validity of 
\be \im(v)\ge I_{\sR^n}(v)-Bv =n B_n^{\frac 1n} v^{\nn}-B v,\qquad 0<v\le \tau_0\label{i1}\ee
for some $B\ge0$ and $\tau_0>0$ guarantees the validity of \eqref{main} for $\mu(\{u\neq0\})\le \tau$, some $\tau\le \tau_0.$


We remark that on a compact Riemannian manifold $(M,g)$ the following expansion holds

\be \im(v)= I_{\sR^n}(v)-\frac{B_n^{-\frac 1n}}{2(n+2)} S v^{1+\frac1n}+o(v^{1+\frac1 n}) \label{i2}\ee

where $S=\max_{x\in M} S_g(x)$, and where $S_g(x)$ is the scalar curvature at $x$. Such an expansion was first derived by Druet \cite{druet-cr}, \cite{druet-pams}. Later, Nardulli showed that  the error term in \eqref{i2} can be replaced with  $O(v^{1+\frac3n})$  on manifolds with cocompact isometry groups \cite{nardulli-agag}, and on manifolds with enough bounded geometry (in which case the maximum defining $S$ has to be taken in the class of pointed limits at infinity of $(M,g)$)
\cite{noa-imrn}, \cite{nardulli-calcvar}. 

We also note that on  complete manifolds satisfying only $\ric \ge K$, the Bishop comparison theorem guarantees that \eqref{i2} holds with  ``$\le$'', and   with $S=S_{g_K^{}}$, the scalar curvature of the space form $(M_K,g_K^{})$ of constant sectional curvature $K$.

It is important to  point out that for small $v$  the error term  \eqref{i1} is larger than the one in \eqref{i2}, however we are able to prove  its uniformity over $M$ under geometric assumptions  which are weaker than the ones given in the aforementioned papers \cite{noa-imrn}, \cite{nardulli-calcvar}.  On compact manifolds, \eqref{i2} was obtained in \cite{druet-pams}  as a consequence of a local version of $I_1^2$ in the form
\be \|u\|_{\nnn}^2\le K(n,1)^2\|\nabla u\|_1^2+K(n,1)^2\Big(\frac n{n+2}S_g(x_0)+\e\Big)\|u\|_1^2\label{i3}\ee
valid for any fixed $x_0\in M$, $\e>0$ and for all $u\in C_c^\infty\big(B(x_0,r_\e)\big)$, some $r_\e>0,$ depending on $x_0$.

Likewise,  on noncompact manifolds with strong uniform $C^3$ conditions on the metric, Nardulli-Osorio Acevedo in \cite{noa-imrn}  derived \eqref{i3}
with $r_\e$ independent of $x_0$, and from this they eventually deduced \eqref{i2}
 (with the error  $O(v^{1+\frac3n})$).
 
 We mention here that in \cite{bm} the authors proved that the best constant $B$ in $I_p^p$, for $1\le p< \min\{2,\sqrt{n}\big\}$, depends continuously on $g$ in the $C^2$ topology on the space of smooth metrics.

In this paper we will derive \eqref{i1} from a local version of $I_1^1$ on manifolds satisfying \eqref{ric}, namely we will show that there are $B,r_0>0$ such that 
\be \|u\|_\nnn\le K(n,1)\big(\|\nabla u\|_1+B\|u\|_1\big)\label{i4}\ee
for any $u\in C_c^\infty(B(x,r_0))$, and any $x\in M$.

For a given $\Omega\subseteq \R^n$ of finite perimeter, applying the above inequality to (smooth approximations of)  $\chi_\Omega$  will yield immediately \eqref{i1} if $\Omega$ is an isoperimetric region of volume $v$, i.e.  if $\im(v)=\per(\Omega)$, $\vol(\Omega)=v$, and provided that the diameter of $\Omega$ is smaller than $r_0.$ This last condition is guaranteed  for $v$ small enough by  results in \cite{noa-imrn}, \cite{apps-annales}. One difficulty to overcome is that on noncompact manifolds isoperimetric regions of given volume may not exist on $M$. However, under our assumptions,  they will exist either on $M$ or on some $C^{1,\alpha}$ pointed limit manifold by results in \cite{anp}, \cite{apps-annales}, \cite{nardulli-asian}, \cite{nardulli-calcvar},  in which case they can   be well  approximated by regions on $M$, to which inequality \eqref{i4} can be applied.

We need to observe that \eqref{i4} was derived globally  on any smooth compact manifold in \cite{druet-gd}, however the methods used there do not provide enough information on the constant $B$, in order to be able to extend the inequality to arbitrary noncompact manifolds satisfying \eqref{ric}.

Our proof of \eqref{i4} is an adaptation of the one given by Druet  of \eqref{i3} in \cite{druet-pams}, and later extended to the case of noncompact manifolds with enough bounded geometry in \cite{noa-imrn}, Proposition 1 (see \cite{noa-arxiv} for its proof). We claim no originality here, other than the fact that we use  harmonic coordinates rather than geodesic normal coordinates. In the context of ``best constants in the  Sobolev  inequalities'' we are aware of only one result where harmonic coordinates were used, i.e. \cite{hebey-ajm} Theorem 1 (see also \cite{hebey-courant} Theorem 7.1). In that result the constant $A$ is almost sharp, i.e. of type $K(n,p)+\e$.

 The geometric assumption \eqref{ric} guarantees the existence of a positive number $\rh$, the harmonic radius, such that on any geodesic ball $B(x,\rh)$ one can find a harmonic coordinate chart where the components of the metric satisfy uniform $C^1$ (in fact $C^{1,\alpha}$) estimates, and which approximate well the components of the Euclidean metric up to first order, which is enough for the purpose of obtaining \eqref{i4}. We point out that while this method works well for $I_1^1$, there seems to be   no hope to obtain for example \eqref{i3}, which requires the coordinate charts to be normal, at the very least, providing  a uniform approximation of the Euclidean metric up to second  order.  Our proof will in fact lead an upper bound of $B$ of type $c_n \rh^{-1}$, for some explicit $c_n$ depending on $n$ and where $\rh$ is the $C^{1,\alpha}$ harmonic radius.

The constant $B$ in \eqref{i4} is the same as the constant $B$ in  \eqref{i1}. For this reason, and given the expansion in \eqref{i2}, it  is reasonable to expect that one can take the  $B$ in \eqref{i4} to be any small number ($r_0$ depending on it) but we do not know how to prove this yet.
Related to this,  we recall a result by Hebey in  \cite{hebey-courant} Prop. 5.1, which shows that on smooth  Riemannian manifolds of dimension $n\ge 4$, the validity of $I_2^2$ implies that the scalar curvature is bounded above by a multiple (depending on $n$) of $B$. Hence, the validity of $I_2^2$ under $\ric\ge K, \inj(M)>0$ is not guaranteed, since it forces  $|\ric|\le K $ (see \cite{hebey-courant}, Theorem 7.3 and its proof). In view of this, it is natural to speculate that $I_p^\a$ may hold under \eqref{ric}, even when $1<\alpha\le 2$, or at least when $1<\a<2$, but we do not know how to prove this with the techniques presented in this paper.

\bigskip
\ni{\bf Acknowledgments.} The authors would like to thank Emmanuel Hebey for helpful comments.
\bigskip
 
\section{Preliminaries}\label{2}

\bigskip
In what follows $(M,g)$ will denote  a smooth, complete, connected,  noncompact, $n-$di-mensional Riemannian manifold without boundary, with metric tensor $g$. The geodesic distance between two points $x,y\in M$ is denoted by $d(x,y)$, and the open geodesic ball centered at $x$ and with radius $r$ will be denoted as $B(x,r)$. The manifold $(M,g)$ is equipped with the natural Riemannian measure $\mu$ which in any chart satisfies $d\mu=\sqrt{|g|} dm$, where $|g|=\det(g_{ij})$, $g=(g_{ij})$ and $dm$ is the Lebesgue measure.

The volume of a measurable set $E\subseteq M$ is defined  as $\vol(E)=\mu(E)$, and its perimeter as the total variation of $\chi_E$, i.e. 
\be \per(E)=\sup\bigg\{\int_E {\rm div} X d\mu: X\in { \mathcal X}_c^1(M),\; \|X\|_\infty\le 1\bigg\},\ee
where ${\mathcal X}_c^1(M)$ denotes the space of $C^1$ vector fields with compact support on $M$, 

If $E$ has finite perimeter, then its reduced boundary $\partial^*E$ is an $(n-1)$-rectifiable set, and up to modifying $E$ on a set  0 measure, it satisfies   $\overline{\partial_{}^*E }=\partial E$, the topological boundary, and $\per(E)$ coincides with the $(n-1)$-dimensional Hausdorff measure (induced by $\mu$) of $\partial^*E$. If $\partial E$ is smooth, then $\partial^*E=\partial E$.

For the definitions and the basic properties of the classical Sobolev spaces $W^{k,p}$ and $W^{k,p}_0$ on Riemannian manifolds
we shall refer the reader to \cite{hebey-courant}, Chapters 2 and 3.

\smallskip
\ni{\bf Convention.} Throughout the paper, the dependence of various objects ($\vol, \per, \nabla,$ etc.) on the ambient metric where these objects are defined will be suppressed, unless ambiguities arise (see  for ex. Remark \ref{amb}).

\smallskip
  The \emph{decreasing rearrangement} of a measurable function $f$ on $M$ is defined as 

\be f^*(t)=\inf\big\{s\ge0:\; \mu\big(\{x:|f(x)|>s\}\big)\le t\big\},\qquad  t>0\label{rearr}\ee
(we will assume that $\mu(\{x:|f(x)|>s\})<\infty$  for all  $s>0$).

For $\tau>0$ we have

\be \mu( \{x:|f(x)|>f^*(\tau)\})\le \tau\le \mu(\{x:|f(x)|\ge f^*(\tau)\})\label{tau}\ee
and $f^*(\tau)>0$ if and only if $\mu(\{f\neq0\})>\tau$, in which  case we can find a measurable set $E_\tau$ such that 
\be\mu(E_\tau)=\tau,\qquad \{x:|f(x)|>f^*(\tau)\}\subseteq E_\tau\subseteq\{x:|f(x)|\ge f^*(\tau)\}.\label{set1}\ee
(see  [MQ, (3.11)] and the comments thereafter). Additionally, we have
\be \int_{E_\tau}\Phi(|f(x)|)d\mu(x)=\int_0^\tau \Phi\big(f^*(t)\big)dt\label{int1}\ee
for any nonnegative measurable $\Phi$ on $[0,\infty)$, and in \eqref{int1} we can take $E_\tau=M$ when $\tau=\infty.$

The \emph{Schwarz symmetric decreasing rearrangement} of a measurable function $f$ on $M$ is the function $f^{\#}:\R^n\to[0,\infty)$ defined as 
\be f^\#(\xi)=f^*(B_n|\xi|^n),\qquad \xi\in\R^n.\ee
\smallskip

The \emph{isoperimetric profile} of $M$ is the function $\im:\big[0,\vol
(M)\big)\to[0,\infty)$ (where $\vol(M)\in (0,\infty]$) defined as 
\be \im(v)=\inf\big\{\per(E),\; E\; {\text {measurable}},\; \vol(E)=v\big\}.\ee
In the definition of  $\im(v)$ the sets $E$ can be chosen with smooth boundary (see \cite{mfn} Thm.1).

The classical isoperimetric inequality on $\R^n$ states that for every measurable set $E$ with finite measure 
\be \frac{\per(E)}{|E|^\nn}\ge \frac{\per (B_{\sR^n}(1))}{|B_{\sR^n}(1)|^\nn}=n B_n^{\frac1n},\ee
where $B_{\sR^n}(1)$ is the unit ball of $\R^n$, hence
\be I_{\sR^n}(v)=n B_n^{\frac1n} v^\nn.\ee
The \emph{normalized isoperimetric profile} of $M$  is defined to be the function 
\be I_0(v)=\frac{\im(v)}{I_{\sR^n}(v)}=\frac{\im(v)}{n B_n^{\frac1n} v^\nn}.\ee

It turns out that on an arbitrary Riemannian manifolds the function $I_M$ is at least upper semicontinuous, and hence measurable (see \cite{mfn}, Corollary 1); clearly the same is true for $I_0$.

If $\ric\ge K$ and the unit ball does not collapse, i.e. if there is $\delta>0$ s.t. $\vol(B(x,1))\ge \delta$ for any $x\in M$, then $\im(v)>0$ for any $v>0$ (see \cite{afp} Remark 4.7), and the function $v\to \im(v)$ is continuous  on $\big(0,\vol(M)\big)$, in fact it is H\"olderian of order $1-\frac1n$ (\cite{mfn} Theorem 2). 

Additionally, if $|\ric|\le K$, the unit ball does not collapse, and the $C^1$ harmonic radius is positive, then 
\be I_0(0^+):=\lim_{v\to0^+} I_0(v)=1.\ee
We note that under \eqref{ric} the non collapsing condition is certainly verified, and the $C^{1,\alpha}$ harmonic radius is positive (see Section \ref{5} for more details on this). Hence, under \eqref{ric}, we can guarantee that $I_0(v)$ is  continuous on  $\big[0,\vol(M)\big)$, and $I_0(0)=1.$

\bigskip

The classical  \emph{P\'olya-Szeg\H o inequality} on $\R^n$ (see e.g. \cite{k}) states that for any open set $\Omega\subseteq \R^n$ and any  $u\in W_0^{1,p}(\Omega)$, $1\le p<\infty$, we have 
\be\int_{\Omega^\#} |\nabla u^\#|^p d\xi\le \int_{\Omega} |\nabla u|^p d\xi,
\ee
where $\Omega^\#$ denotes the Euclidean ball centered at $0$ and with volume $\vol(\Omega)$.

Note that several authors define the decreasing rearrangement as in \eqref{rearr} but using the level sets $\{x: f(x)>s\}$, in which case
the P\'olya-Szeg\H o inequality must be stated for nonnegative functions.

\smallskip
On an arbitrary smooth, connected, noncompact  Riemannian manifold a P\'olya-Szeg\H o inequality exists in the following form: 

\begin{prop}\label{polyaszego} Let  $\Omega\subseteq M$ be open with  $\vol(\Omega)< \infty$.  Then, for any $u\in C_c^\infty(\Omega)$, its Schwarz symmetrization  $u^\#$ is Lipschitz on $\Omega^\#$, moreover for any $p\in[1,\infty)$ we have $u^\#\in W_0^{1,p}(\Omega^\#)$, and 
\be\int_{\Omega^\#} I_0(B_n |\xi|^n)^p |\nabla u^\#(\xi)|^p d\xi\le \int_{\Omega} |\nabla u(x)|^p d\mu(x).\label{ps}
\ee
\end{prop}

The above result does  not appear explicitly in the literature, but it can be obtained from standard methods via the co-area formula, H\"older's inequality (for $p>1$),  and the very definition of the isoperimetric profile. See for example the proof of Lemma 2.1  in \cite{kr}, which is stated  for $p=n$ under the assumption $\mathrm{Isop}(\Omega,g)>0$ (which is implied for example by $I_0\ge\kappa>0$ on $(0,\mathrm{Vol}(\Omega))$),  but the proof works for any $p\ge1$. The  inequality \eqref{ps} can be obtained from that proof  by replacing Isop$(\Omega,g)$ with $I_0(V(t))$ in the first inequality of page 2373, and then using the co-area formula.  We note that the argument in \cite{kr} implicitly uses the regularity of the rearranged function (namely that it is Lipschitz, or at least in $W^{1,p}_0(\Omega^\#)$). We also note that a version of \eqref{ps}  is stated in \cite{apps-mathann} Prop 3.10 (they also consider RCD$(0,N)$ spaces which technically do not include our spaces, and they only consider $p>1$). 
An  outline of the proof of Proposition \ref{polyaszego} is given in the Appendix.

\vskip1em
\section{Sobolev inequalities: from small volumes to global}\label{3}

Let $(M,g)$ be a complete, $n-$dimensional  Riemannian manifold, without boundary, and let $\mu$ be  the associated Riemannian measure. 

For the rest of this paper we will let 
\be 1\le p<n,\qquad \quad q=\frac{np}{n-p}.\label{qp}\ee
 We will say that given $A>0$ the  \emph{Sobolev inequality with constant $A$ holds for small volumes on $M$}, if there are $\tau>0$ and  $B\ge 0$ such that for any open set $\Omega$ with $\vol(\Omega)<\tau$, and for all $u\in W_0^{1,p}(\Omega)$  we have
\be \|u\|_q\le A\|\nabla u\|_p+B\|u\|_p.\label{sob}\ee
Likewise, we will say that the \emph{Sobolev inequality with constant $A$ holds globally on $M,$} if there is $B\ge0$ such that \eqref{sob} holds for all $u\in W^{1,p}(M).$

\smallskip
\begin{theorem}\label{loc-glob} Given $p,q$ as in \eqref{qp}, and $A>0$, if the Sobolev inequality with constant $A$ holds for small volumes on $M$, then it holds globally.

\end{theorem}

\ni{\bf Proof.}  It is enough to prove \eqref{sob} for $u\in C_c^\infty(M)$. If  $u=u^+-u^-$, then $u^+=u\chi_{\{u>0\}}\in W_0^{1,p}(\{u>0\})$ and $\nabla u^+=\nabla u\chi_{\{u>0\}}$, and similar identities hold for $u^{-}$. Using the elementary inequality 
\be \bigg(\sum_{j=1}^N a_j^q\bigg)^{\frac 1q}\le \bigg(\sum_{j=1}^N a_j^p\bigg)^{\frac 1p},\qquad q>p,\qquad a_j\ge0,\qquad N\in\N\label{ineq}\ee
it is easy to see that without loss of generality we can assume  $u\ge0.$ Indeed, if that is the case then
\be\ba \|u\|_q&=\big(\|u^+\|_q^q+\|u^-\|_q^q\big)^{\frac 1q}\le\Big(\big(A\|\nabla u^+\|_p+B\|u^+\|_p\big)^q+\big(A\|\nabla u^-\|_p+B\|u^-\|_p\big)^q\Big)^{\frac 1q}\cr&
\le\Big(A^q\|\nabla u^+\|_p^q+A^q\|\nabla u^-\|_p^q\Big)^{\frac1q}+\Big(B^q\| u^+\|_p^q+B^q\|u^-\|_p^q\Big)^{\frac1q}\cr&
\le \Big(A^p\|\nabla u^+\|_p^p+A^p\|\nabla u^-\|_p^p\Big)^{\frac1p}+\Big(B^p\| u^+\|_p^p+B^p\|u^-\|_p^p\Big)^{\frac1p}=A\|\nabla u\|_p+B\|u\|_p.
\ea\ee
Assume then $u\ge0$. If $\mu(\{u>0\})\le \tau$ there is nothing to prove, so we can assume $\mu(\{u>0\})>\tau$, i.e. $u^*(\tau)>0$. Write
\be u=\big(u-u^*(\tau)\big)^++u^*(\tau)\chi_{\{u\ge u^*(\tau)\}}+u\chi_{\{u<u^*(\tau)\}}.\ee
and using Minkowski's inequality we get
\be\|u\|_q\le \|\big(u-u^*(\tau)\big)^+\|_q+u^*(\tau)\mu\big(\{u\ge u^*(\tau)\}\big)^{\frac1q}+\bigg(\int_{\{u<u^*(\tau)\}} u^q d\mu\bigg)^{\frac1q}.
\ee
From \eqref{tau} and the fact that $q>p$ we have
\be\ba u^*(\tau)\mu\big(\{u\ge u^*(\tau)\}\big)^{\frac1q}&\le \tau^{\frac1q-\frac1p}u^*(\tau)\mu\big(\{u\ge u^*(\tau)\}\big)^{\frac1p}\cr&=\tau^{-\frac1n}\bigg(\int_{\{u\ge u^*(\tau)\}} u^*(\tau)^p d\mu\bigg)^{\frac1p}\le \tau^{-\frac1n}\|u\|_p.\ea\ee

From \eqref{set1} and \eqref{int1} we obtain
\be\int_{\{u<u^*(\tau)\}} u^q d\mu=\int_M u^q d\mu -\int_{\{u\ge u^*(\tau)\}} u^q d\mu\le \int_\tau^\infty \big(u^*(t)\big)^qdt\ee
and  therefore, using \eqref{ineq}
\be\ba\bigg(\int_{\{u<u^*(\tau)\}} u^q d\mu\bigg)^{\frac1q}&\le \bigg(\int_\tau^\infty \big(u^*(t)\big)^qdt\bigg)^{\frac1q}=
\bigg(\sum_{k=1}^\infty\int_{k\tau}^{(k+1)\tau} \big(u^*(t)\big)^qdt\bigg)^{\frac1q}\cr&\le
\bigg(\sum_{k=1}^\infty\tau\big(u^*(k\tau)\big)^q\bigg)^{\frac1q}\le \tau^{\frac1q-\frac1p}\bigg(\sum_{k=1}^\infty\tau\big(u^*(k\tau)\big)^p\bigg)^{\frac1p}\cr&\le  \tau^{\frac1q-\frac1p}\bigg(\sum_{k=1}^\infty\int_{(k-1)\tau}^{k\tau} \big(u^*(t)\big)^pdt\bigg)^{\frac1p}=
\tau^{-\frac1n}\|u\|_p.
\ea\ee
Since $\big(u-u^*(\tau)\big)^+\in W_0^{1,p}(\{u-u^*(\tau)>0\})$ and $\mu\big(\{u-u^*(\tau)>0\}\big)\le \tau$, under the given assumption of  Sobolev inequality for small volumes we obtain
\be\ba \|u\|_q&\le A\|\nabla \big(u-u^*(\tau)\big)^+\|_p+B\|\big(u-u^*(\tau)\big)^+\|_p+2\tau^{-\frac1n}\|u\|_p
\cr&\le A\|\nabla u\|_p+(B+2\tau^{-\frac1n})\|u\|_p.\ea
\ee
\rightline\qed

\vskip1em
\section{From isoperimetric profile expansions to Sobolev inequalities for small volumes}\label{4}

\medskip
In this section we show how a first order expansion of the isoperimetric profile for small volumes, implies the sharp Sobolev inequality for small volumes.

\begin{theorem}\label{smallprofile-sobolev}
Assume that on a  complete Riemannian manifold   there exist constants $\tau_0>~0$, with $\tau_0<\vol(M)$,  and $B\ge0$ such that  $I_0(0^+)=1$ and 
\be I_0(v)\ge 1- B v^{\frac 1n},\qquad 0<v\le \tau_0.\label{cond}\ee
Then the Sobolev inequality \eqref{sob} with constant $K(n,p)$ holds for small volumes on $M$, in particular, there is $\tau\le \tau_0$ such that for any $\Omega$ open with $\vol(\Omega)\le \tau$
\be \|u\|_q\le K(n,p)\|\nabla u\|_p+4B\big(1+ B_n^{\frac 1n} K(n,p)\big)\|u\|_p,\qquad u\in W_0^{1,p}(\Omega).\label{sob3} \ee
\end{theorem}

\ni {\bf Proof.}
It is enough to prove \eqref{sob3} for  $u\in C_c^\infty(\Omega)$, with $\Omega$ open with $\vol(\Omega)\le \tau$. Let us pick $\tau\le \tau_0$ so small so that 
\be1-B\tau^{\frac1n}\ge \frac12,\qquad I_0(v)\le 2, \qquad {\text {for}} \;\;0\le v\le \tau,\label{tt}\ee
where for the second condition we use the fact that $I_0(0^+)=1$.

Since $u\in C_c^\infty(\Omega)$ and  $\vol(\Omega)\le\tau<\vol(M)$, from Proposition \ref{polyaszego}  we have that  $u^\#$ is Lipschitz, and supported  in $\Omega^\#\subseteq B_{\sR^n}(0,R)$,  where $B_n R^n=~\tau$. Hence,
\be\ba &u^\#(\xi)=\int_{|\xi|}^R (-\partial_\rho) u^\#(\rho\xi^*)d\rho=\int_{|\xi|}^RI_0(B_n\rho^n)^{-1}I_0(B_n\rho^n) (-\partial_\rho) u^\#(\rho\xi^*)d\rho\cr&\le \int_{|\xi|}^R(1+2BB_n^{\frac 1n}\rho) I_0(B_n\rho^n) (-\partial_\rho) u^\#(\rho\xi^*)d\rho
\le  v(\xi)+4BB_n^{\frac 1n}\int_{|\xi|}^R \rho (-\partial_\rho) u^\#(\rho\xi^*)d\rho\cr&=v(\xi)+4BB_n^{\frac 1n}|\xi| u^\#(\xi)+4BB_n^{\frac 1n}\int_{|\xi|}^R  u^\#(\rho\xi^*)d\rho
\ea
\ee
where we set 
\be v(\xi)=\int_{|\xi|}^R I_0(B_n\rho^n)(-\partial_\rho)u^\#(\rho\xi^*)d\rho,\ee
and where the  two inequalities above are easy consequences of \eqref{cond}, \eqref{tt}. 

From the \PS inequality \eqref{ps} we  have
\be\ba \int_{B(0,R)}|\nabla v(\xi)|^p d\xi&=\int_{B(0,R)}I_0(B_n|\xi|^n)^p|\nabla u^\#(\xi)|^p d\xi \le\int_M|\nabla u(x)|^p d\mu(x).\label{psv}\ea\ee
From Minkowski's inequality
\be\ba\|u\|_q&= \bigg(\int_{B(0,R)} \big(u^\#(\xi)\big)^qd\xi \bigg)^{\frac1q}\le\bigg(\int_{B(0,R)}\big(v(\xi)\big)^qd\xi \bigg)^{\frac1q}+4BB_n^{\frac 1n}\bigg(\int_{B(0,R)}|\xi|^{q} \big(u^\#(\xi)\big)^qd\xi \bigg)^{\frac1q}\cr&+4BB_n^{\frac 1n} \Bigg(\int_{B(0,R)} \bigg(\int_{|\xi|}^R u^\#(\rho\xi^*)d\rho\bigg)^q d\xi \Bigg)^{\frac1q}:=I+II+III\ea\ee
From \eqref{psv} and the sharp Sobolev inequality on $\R^n$ we have
\be I=\|v\|_q\le K(n,p)\|\nabla v\|_p\le K(n,p)\|\nabla u\|_p.\label{I}\ee

To estimate the second term, first note that $u^\#(\xi)\le \|u\|_p B_n^{-\frac1 p}|\xi|^{-\frac n p }$, since for any $\xi\in B(0,R)$
\be \|u\|_p\ge \bigg(\int_{B(0,|\xi|)}\big(u^\#(\eta)\big)^pd\eta\bigg)^{\frac 1p}\ge u^\#(\xi) (B_n|\xi|^n)^{\frac 1p}.\ee
Hence,
\be\ba&\int_{B(0,R)}|\xi|^{q} \big(u^\#(\xi)\big)^qd\xi =\int_{B(0,R)}|\xi|^{q} \big(u^\#(\xi)\big)^{q-p}\big(u^\#(\xi)\big)^pd\xi 
\cr&\le \int_{B(0,R)}|\xi|^{q} \big(\|u\|_pB_n^{-\frac1p}|\xi|^{-\frac np}\big)^{q-p}\big(u^\#(\xi)\big)^pd\xi \cr& =B_n^{1-\frac q p}\|u\|_p^{q-p}\int_{B(0,R)}|\xi|^{q-\frac {nq} p+n} \big(u^\#(\xi)\big)^pd\xi =B_n^{-\frac q n}\|u\|_p^{q-p}\int_{B(0,R)} \big(u^\#(\xi)\big)^pd\xi\cr&= B_n^{-\frac q n}\|u\|_p^q \ea\ee
so that 
\be II\le 4B\|u\|_p.\ee

\smallskip
To estimate $III$, we just observe that the function on $\R^n$ 
\be F(\xi)=\int_{|\xi|}^R  u^\#(\rho \xi^*)d\rho\ee
is in $W_0^{1,p}(B(0,R))$, and $|\nabla F|=u^\#$ (note that $u^\#(\rho \xi^*)$ only depends on $\rho$), hence from the Sobolev inequality \eqref{I} applied to $F$  we get 
\be III= 4B B_n^{\frac 1n} \|F\|_q\le 4B B_n^{\frac 1n}  K(n,p)\|\nabla F\|_p=4B B_n^{\frac 1n} K(n,p)\|u\|_p.\ee \qed

\begin{rk} From the above proof  it's clear that the factor $4B$ in \eqref{sob3} can be reduced to $2B(1+\e)$ for a suitable $\tau$ depending on $\e$ (modify the right-hand side of the second estimate in \eqref{tt} to $1+\e$).
\end{rk}

\begin{rk} As it  apparent in the above  proof, the conclusion of Theorem 3 holds assuming only  
\be\frac12\le 1-Bv^{\frac1n}\le I_0(v)\le 1+\delta,\qquad 0< v\le \tau\ee
for some $\delta, \tau>0$.
\end{rk}

\vskip1em
\section{A local Sobolev inequality for $W^{1,1}$}\label{5}
For $f:\Omega\to\R$, $\Omega$ open and bounded on $\R^n$, define
\be\|f\|_{C(\Omega)}=\sup_{\xi\in \Omega}|f(\xi)|\ee
and for $\alpha\in (0,1]$ define the seminorm
\be |f|_{C^{0,\a}(\Omega)}=\sup_{\xi,\eta\in \Omega\atop \xi\neq\eta}\frac{|f(\xi)-f(\eta)|}{|\xi-\eta|^\a}.\ee

We recall some basic definitions about harmonic radius, mainly taken from \cite{hebey-herzlich}.
Given a smooth Riemannian manifold without boundary, let $k\in \Z,\; k\ge0$, $\alpha\in (0,1)$, $Q>1$, and $x\in M$. The \emph{$C^{k,\alpha}$ harmonic radius  at $x$}, denoted as $\rh(g,Q,k,\alpha,x)$ is the largest $r$ such that on the geodesic ball $B(x,r)$ there is a harmonic  coordinate chart $\Phi:U\to\R^n$, with $B(x,r)\subset U $, and coordinates $\xi=(\xi_1,...,\xi_n)=\Phi(y)$, with $\Phi(x)=0$ and such that the components $g_{ij}=g(\partial_{\xi_i},\partial_{\xi_j})$ of $g$ in such coordinates, seen as functions on $B(x,r)$, satisfy

\smallskip
\be g_{ij}(x)=\delta_{ij} \label{h1} \tag{H1}\ee

\be Q^{-1}\delta_{ij}\le g_{ij}\le Q \delta_{ij} \quad {\text{as bilinear forms}} \label{h2} \tag{H2}\ee

\be r^{|\beta|}\|\partial_\b g_{ij}\|_{C(B(x,r))}\le Q-1,\quad  1\le |\beta|\le k\label{h3} \tag{H3}\ee

\be r^{k+\a} |\partial_\b g_{ij}|_{C^{0,\a}(B(x,r))}\le Q-1,\quad  |\beta|= k\label{h4} \tag{H4}\ee

\smallskip
with the provision that if $\alpha=0$ condition \eqref{h4} is omitted, and if $k=0$ then condition \eqref{h3} is omitted and \eqref{h4} holds for $k=0$. In \eqref{h3}, \eqref{h4} we set for simplicity $C(B(x,r))=C(\Phi(B(x,r)))$ and $C^{0,\a}(B(x,r))=C^{0,\a}(\Phi(B(x,r))$).

With the above notation the local Euclidean  behavior of the metric is captured as $Q\to1$, as well as the  fact that if for fixed $Q>1$, and some $x\in M$, we have  $\rh(Q,k,\alpha,x)=+\infty$ then the metric $g$ is flat.
Also, note that conditions \eqref{h1}-\eqref{h4} are invariant under dilation of $g$ (change $g$ to $\lambda^2 g$ then consider the chart $\lambda\Phi$, so the components of the new metric in that chart are $g_{ij}(\xi/\lambda)$, and the corresponding $\rh$ is $\lambda \rh$).

Note that the distance between $y,z\in B(x,\rh)$ under the metric $g$ is related to the Euclidean distance of the coordinates $\xi=\Phi(y), \eta=\Phi(z)$ via
\be Q^{-\frac12}|\xi-\eta|\le d(y,z)\le Q^{\frac12}|\xi-\eta|\label{dist1}\ee
as it's clear from \eqref{h2}. As a consequence, we obtain that for any $ z\in B(x,\rh)$ and  $r<\rh-d(x,z)$ (in particular when $d(x,z)<\rh/2,\; r<\rh/2)$
\be B_{\sR^n}(\eta,rQ^{-\frac12})\subseteq \Phi\big(B(z,r)\big)\subseteq B_{\sR^n}(\eta,rQ^{\frac12}).\label{balls}\ee

\medskip
For fixed $Q>1$, the $C^{k,\alpha}$ \emph{harmonic radius of $M$} is defined as 
\be \rh=\rh(g,Q,k,\a)=\inf_{x\in M} \rh(g,Q,k,\a,x).\ee

A fundamental result is that given $k\in\Z, k\ge0$, $\alpha\in (0,1)$, $Q>1$, then under the assumptions
\be |\nabla ^j \ric|\le K,\qquad \inj(M)>0,\qquad j=0,1,..,k\ee
one can guarantee that  $\rh(g,Q,k+1,\a)>0$. For $k=0,1$ these are the classical results by \cite{ac}, \cite{and}. See \cite{hebey-herzlich}  Theorem 6 for the general case.

\begin{lemma} \label{volest} Assume $Q>1$, $\rh>0$, $k=0$ and $\alpha\in(0,1)$, or $k=1$ and $\alpha=0$. If $x\in M$ and 
\be0<\delta_H<\min\Big\{\rh Q^{-\frac12}, \rh\big(n^2(Q-1)\big)^{-\frac1{k+\a}}\Big\}\label{deltah}\ee
 then in a harmonic chart where \eqref{h1}-\eqref{h4} hold with $r=\rh$ we have
 \be |g_{ij}(\xi)-\delta_{ij}|\le\rh^{-k-\a}(Q-1)|\xi|^{k+\a}.\label{gij10}\ee
 \be |g^{ij}(\xi)-\delta_{ij}|\le 2n\rh^{-k-\a}(Q-1)|\xi|^{k+\a}.\label{gij20}\ee
\be\big|\sqrt{|g|(\xi)}-1\big|\le n^2 \rh^{-k-\a}(Q-1) |\xi|^{k+\a},\qquad |\xi|<\delta_H.\label{detg0}\ee
\end{lemma} 

\ni{\bf Proof.} First note that \eqref{balls} guarantees that $B_{\sR^n}(0,\delta_H)\subseteq B_{\sR^n}(0,r_HQ^{-\frac12}) \subseteq \Phi\big(B(x,r_H)\big)$. Estimate \eqref{gij10} follows from the mean value theorem when $k=1,\;\alpha=0$, and from the definition of $|\cdot|_{C^{0,\a}}$ when $k=0,\; \a\in(0,1)$.

Considering the matrices $G=(g_{ij})$ and $H=(h_{ij}):=G-I$, then $$|h_{ij}|\le h:= \rh^{-k-\a}(Q-1)|\xi|^{k+\a}
<\frac1{n^2},$$ so Weyl's Theorem (see \cite{horn-johnson}  Theorem 4.3.1) implies that 
\be\lambda_k(I)+\lambda_1(H)\le \lambda_k(G)\le \lambda_k(I)+\lambda_n(H),\qquad k=1,2,...,n\ee
where $\lambda_k(A)$ denotes the $k-$th largest eigenvalue of a symmetric matrix $A$. It's easy to see that  $\lambda_1(H)\ge -nh$ and $\lambda_n(H)\le nh$
which gives
\be (1-nh)^{n/2}\le \sqrt{\det G}\le (1+nh)^{n/2}.\ee
Estimate \eqref{detg0} follows from the inequalities 
\be1-a{|x|}\le (1+x)^a\le 1+2a |x| , \qquad |x|<\frac1 a,\; a>1\ee
with $a=n/2$, and $x=nh<1/n$, together with \eqref{dist1}.

Finally, by  a standard result in matrix theory (see 
\cite{horn-johnson}  (5.8.2))
\be\max_{ij}|g^{ij}-\delta_{ij}|\le  \|(I+H)^{-1}-I\|_2\le\frac{\|H\|_2}{1-\|H\|_2}\le \frac {nh}{1-nh}\le 2nh,\ee 
(since $h<1/n^2\le 1/(2n))$ which gives \eqref{gij20}.

\rightline\qed

\smallskip
\begin{theorem}\label{druet} Assume $|\ric|\le K$ and $\inj(M)>0$. Then, there exist $B>0$  and $r_0>0$ such that for any $x\in M$ and  all $u\in W_0^{1,1}(B(x,r_0))$
\be \|u\|_\nnn\le K(n,1)\big(\|\nabla u\|_1+B\|u\|_1\big).\label{sobdruet}
\ee
\end{theorem}

\smallskip
\begin{rk} \label {QQ} From the proof below it will be apparent that we can take  $B=C(n,Q) r_H^{-1}$, for any fixed $Q>1$ and for some explicit $C(n,Q)$, depending only on $n$ and $Q$.   For simplicity we will only consider $Q=4$.
\end{rk}

\smallskip

\begin{cor} \label{druet1} Assume $|\ric|\le K$ and $\inj(M)>0$. Then, there exists $B>0$ and $r_0>0$ such that for any set $\Omega\subseteq M$ of finite perimeter and with $\diam (\Omega)\le r_0$ we have
\be \per(\Omega)\ge nB_n^{\frac1n} \vol(\Omega)^{\frac{n-1}n}-B\;\vol(\Omega).
\label{isop1}\ee
\end{cor}

The proof of this corollary is immediate from Theorem \ref{druet}, since such $\Omega$ must be included in some ball $B(x_0,{\frac23}r_0)$ of radius $\frac23r_0$, up to a set of 0 measure, and, by standard results, one can approximate $\chi_\Omega$ by a sequence $u_k\in C_c^\infty(B(x_0,r_0))$, so that $\|u_k\|_\nnn\to \vol(\Omega)^\nn$, and $\|\nabla u_k\|_1\to \per(\Omega).$ Applying \eqref{sobdruet} to the $u_k $ and passing to the limit yields \eqref{isop1} (recall that $n B_n^{\frac1n}=K(n,1)^{-1}$).

\smallskip
\ni{\bf Proof of Theorem \ref{druet}.}  For $p\ge 1$, $B>0$, $r>0$, and $x\in M$, let
\be \lambda_{p,r}(x)=\lambda_{p,r,g}(x)=\inf_{u\in W_0^{1,1}(B_g(x,r))\atop u\not\equiv 0}\frac{\|\nabla u\|_p+B \|u\|_p}{\|u\|_q}.\label{lambda}\ee

Fix any $B>0$, and assume that there is no $r_0>0$   so that the conclusion of Theorem~\ref{druet} is true. We will then derive an explicit upper bound on $B$ depending on $n$ and $\rh$, see \eqref{B}.
Arguing  as in   \cite{druet-pams}, \cite{noa-arxiv},  for all $r>0$  there is $y_r\in  M$ such that 
\be \lambda_{1,r}(y_r)<K(n,1)^{-1}.\ee

Since for fixed $r$, and any $x$,  we have $\limsup_{p\to1}\lambda_{p,r}(x)\le \lambda_{1,r}(x)$, then picking $r_k\searrow 0$,  there is $p_k\searrow 1$ such that for each $k$ large enough 
\be \limsup_{p\to 1}\lambda_{p,r_k}(y_{r_k}^{})-\frac1k<\lambda_{p_k,r_k}(y_{r_k}^{})<K(n,1)^{-1}\ee

For the  rest of this proof  ``$p$'' will denote an element of the original  sequence $p_k\searrow 1,$ or a subsequence of it, and   the notation $p\to1$ will always mean ``up to a subsequence of $\{p_k\}$''. To be consistent with the notation used in \cite{druet-pams} and \cite{noa-arxiv}, given any subsequence of $\{p_k\}$ we will use $r_p$ to denote the corresponding subsequence  of $r_k\searrow 0$.

\smallskip
We will let 

\be\lambda_p=\ \lambda_{p,r_p}(y_p)\ee

and we can assume that as $p\to 1$

\be \lambda_p< K(n,1)^{-1} \frac{n-p}{p(n-1)}\label{L1}\ee
and, as a consequence,
\be \lambda_p<K(n,p)^{-1}.\label{L2}\ee

\medskip
The fact that the right hand side of \eqref{L1} is smaller than $K(n,p)^{-1}$ was implicitly used in \cite{druet-pams} and \cite{noa-arxiv} without proof, but it follows ``easily'' from 
\be K(n,p)=(p-1)^{1-\frac1p}n^{1-\frac1p}(n-p)^{-1+\frac1p}\bigg[\frac{\Gamma(n)}{\Gamma\big(\frac np\big)\Gamma\big(n+1-\frac np\big)}\bigg]^{\frac1n}\;K(n,1)\label{knp}\ee
combined with the following estimate for the ratio of gamma functions (see e.g. \cite{gm})
\be\frac{\Gamma(x+y+1)}{\Gamma(x+1)\Gamma(y+1)}\le \frac{(x+y)^{x+y}}{x^x y^y},\qquad x,y>0\ee
with $x=\frac np-1.\; y=n-\frac np$.

\begin{rk}\label{kn1p}
We  note that it is alternatively possible to deduce \eqref{L2}, for $p$ close enough to 1, from $\lambda_p<K(n,1)^{-1}$ and  the expansion 
\be \frac{K(n,p)}{K(n,1)}=1+(p-1)\big(\log(p-1)+O(1)\big),\ee
(where $|O(1)|\le C_n$ if $p$ close enough to 1),  and which can be ``easily'' derived from \eqref{knp}.

\end{rk}
\medskip

 From \eqref{L2} using \cite{aubin-li}  Prop. 8.1, combined with regularity results (see for ex. \cite{druet-prse}, Theorem 2.3)  we have the existence of a minimizer $u_p$ (of the functional in \eqref{lambda}) such that

\be u_p\in C^{1,\eta}\big(B(y_p,r_p)\big), \quad {\text {some} }\; \;\eta>0\ee

\be u_p>0 \;\;{\text {on} }\; \;B(y_p,r_p),\qquad u_p=0 \;\;{\text {on} }\;\;  \p B(y_p,r_p)\ee

\be -\Delta_p u_p+B\|u_p\|_p^{1-p}\|\nabla u_p\|_p^{p-1}u_p^{p-1}=\lambda_p\|\nabla u_p\|_p^{p-1}u_p^{q-1}\label{E1}\ee

\be \|u_p\|_q=1,\label{u1}\ee

where $\Delta_p={\rm {div} }\big(|\nabla u|^{p-2}\nabla u\big)$ denotes the $p$-Laplacian in the metric $g$.

Let $x_p$ be the maximum point for $u_p$, and define $\mu_p$ as follows

\be\mu_p^{1-n/p}:= u(x_p)=\max u_p.\ee

We summarize here some basic properties that we will use later:

\begin{prop}\label{properties1} In the above notation the following hold:

\be \lim_{p\to1}\mu_p=0,\label{p1}\ee

\be \lim_{p\to1} \|u_p\|_p=0,\label{p2}\ee

\be \lim_{p\to1}\lambda_p=\lim_{p\to1}\lambda_{p,r_p}(y_p)=K(n,1)^{-1},\label{p3}\ee

\be \lim_{p\to1} \|\nabla u_p\|_p=K(n,1)^{-1}.\label{p4}\ee

\end{prop}
\ni{\bf Proof of Proposition 1.}
Properties \eqref{p1}, \eqref{p2} are easy consequences of 
\be 1=\|u_p\|_q^q\le \vol(B(y_p,r_p))\mu_p^{-n},\ee
\be \|u_p\|_p\le \|u\|_q \vol(B(y_p,r_p))^{1-p/q}= \vol(B(y_p,r_p))^{1-p/q}.\ee
For \eqref{p3} use that for each $\e>0$ there is $B_\e>0$ s.t. for each $v\in W^{1,1}(M)$
\be \|v\|_{\frac n{n-1}}\le (K(n,1)+\e)\|\nabla v\|_1+B_\e\|v\|_1\ee
(\cite{hebey-courant} Theorem 7.1).
Apply the above  to $v=u_p^{\frac {p(n-1)}{n-p}}=u^{1+\frac{n(p-1)}{n-p}}$ so that
\be\ba 1&=\|u_p\|_q\le (K(n,1)+\e)\frac{p(n-1)}{n-p} \int_M|\nabla u_p| u_p^{\frac{n(p-1)}{n-p}}d\mu+B_\e\|u_p\|_p\cr&\le  (K(n,1)+\e)\frac{p(n-1)}{n-p}\|\nabla u_p\|_p+B_\e\|u_p\|_p\cr&=(K(n,1)+\e)\frac{p(n-1)}{n-p}(\lambda_p-B\|u_p\|_p)+B_\e\|u_p\|_p,\ea\label{eq1}\ee
where we used H\"older's inequality, \eqref{u1},   and  (from \eqref{E1})
\be \|\nabla u_p\|_p+B\|u_p\|_p=\lambda_p.\label{eq2}\ee
Clearly \eqref{eq1} and \eqref{p2} imply $\liminf_{p\to1} \lambda_p(K(n,1)+\e)\ge1$ for all $\e$, hence $\liminf \lambda_p\ge K(n,1)^{-1}$, which gives \eqref{p3} using \eqref{L1} (or \eqref{L2}).

Property \eqref{p4} follows from \eqref{eq2} and the result just obtained. 

\qed

\medskip
Now we proceed with the  classical construction of the blowup at the maximum points $x_p$, but  in harmonic coordinates rather than the usual normal coordinates. Fix  $Q=4$ (see Remark  \ref{QQ}), and  let $\Phi_p:B(x_p,\rh)\to \R^n$ be a harmonic chart satisfying \eqref{h1}-\eqref{h4}, with $k=1, \alpha=0$, and with $\rh>0$.  The existence of such chart is guaranteed by the aforementioned classical result by Anderson \cite{and}. Denote the coordinates in the chart by  $\xi=\Phi_p(y)$, with $\Phi_p(x_p)=0$, and  let 
(with some  abuse of language) $u_p(\xi)$ be $u_p(\Phi_p^{-1}(\xi))$, and $g_{ij}^{(p)}(\xi)$ be the components of $g$ in the coordinate chart, where $\xi\in \Phi_p(B(x_p,\rh)).$ 

From \eqref{balls} we have
  \be B_{\sR^n}(0,\half \rh)\subseteq \Phi_p\big(B(x_p,\rh)\big)\subseteq B_{\sR^n}(0,2\rh).\label{balls0}\ee
If  $g_{(p)}^{ij}$ denote the component of the inverse matrix of $\{g_{ij}^{(p)}\}$ and if  
\be0<\dh<\frac\rh {3n^2}\label{deltah}\ee
 then, \eqref{h1}-\eqref{h3} and  Lemma \ref{volest} give,  for $|\xi|<\dh$, 
 
\be g_{ij}^{(p)}(0)=\delta_{ij}, \ee
\be \frac14\delta_{ij}\le g_{ij}^{(p)}\le 4 \delta_{ij} \quad {\text{as bilinear forms,}}\label{h2a}  \ee
\be |\partial_k g_{ij}^{(p)}(\xi)|\le 3\rh^{-1},\label{gij0} \ee
  \begin{align} |g_{ij}^{(p)}(\xi)-\delta_{ij}|&\le3\rh^{-1}|\xi|,\label{gij1}\\
 |g_{(p)}^{ij}(\xi)-\delta_{ij}|&\le 6n\rh^{-1}|\xi|,\label{gij2}\\
\big|\sqrt{|g_{(p)}|(\xi)}-1\big|&\le3 n^2 \rh^{-1} |\xi|,\label{detg}\end{align}
where $|g_{(p)}|=\det\big(g_{ij}^{(p)}\big)$.  Additionally, if we put the metric $g_{(p)}$ on $B_{\sR^n}(0,\frac12\rh)$, then \eqref{balls} becomes
\be B_{\sR^n}(\eta,\half r)\subseteq B_{g_{(p)}}(\eta,r)\subseteq B_{\sR^n}(\eta,2r),\qquad |\eta|<\dh,\; r\le\dh\label{balls1}\ee
where  $B_{g_{(p)}}$ denotes a ball in the metric $g_{(p)}$.

\medskip

 For $p$ close enough to 1 we let
\be \Omega_p=\mu_p^{-1} \Phi_p(B(y_p,r_p)).\ee
Clearly (recall that $d(x_p,y_p)<r_p\to0$)
\be \Omega_p\subseteq \mu_p^{-1} \Phi_p(B(x_p,2r_p))\subseteq B_{\sR^n}(0,4 \mu_p^{-1}r_p)\subseteq B_{\sR^n}(0,\mu_p^{-1}\dh).\label{omegap0}\ee
 Let 
(with some  abuse of language) $u_p(\xi)$ be $u_p(\Phi_p^{-1}(\xi))$, for $\xi\in \Phi_p(B(y_p,r_p))$, and 
\be v_p(\xi)=\bc\mu_p^{{\frac np}-1}u_p(\mu_p\xi ) & {\text{ if  }}\xi\in \Omega_p\\ 0 & {\text{ if }} \xi\in \R^n\setminus \Omega_p.\ec\ee

On $B_{\sR^n}(0,\mu_p^{-1}\dh)$, in particular on $\Omega_p$,  we consider the metric defined on $M$ as $\gp=\mu_p^{-2}g$, which in the coordinate chart given by $\mu_p^{-1}\Phi_p$,   is simply the tensor with components $g_{ij}^{(p)}(\mu_p\xi)$.  In the sequel $dV=dV_g$ denotes the volume element in the original metric $g$ on $M$, while the volume element in the metric $\gp$  on $\Omega_p$ will be denoted as
\be dV_\gp(\xi)=\sqrt{|g^{(p)}|(\mu_p \xi)}\;d\xi\label{dvp},\qquad |g^{(p)}|=\det\big(g_{ij}^{(p)}\big).\ee

\smallskip
\begin{rk}\label{amb}
In order to minimize notation a bit we will only emphasize the role of the metric $\gp$ when necessary. For 
example $\|\cdot\|_{p,\gp}$ will denote the $L^p$ norm of a functions defined on $\Omega_p$  with respect to the measure in \eqref{dvp} while $\|u_p \|_p$  will continue to denote the norm of $u_p$ w.r. to the original metric $g$. Similarly, for a function $v$ defined on $\Omega_p$, the notation $|\nabla_\gp v|_{\gp}$  will mean that the gradient and the inner product are w.r. to $\gp$, while $|\nabla v|$ is just the Euclidean norm of the Euclidean gradient.
\end{rk}

\smallskip

Note that  if  $\xi \in\Omega_p$ then  $\mu_p\xi\in\Phi_p\big( B(y_p,r_p)\big)\subseteq\Phi_p\big( B(x_p,2r_p)\big)$, so, using \eqref{balls}, 
\be\mu_p|\xi|\le 2 d(x_p,\Phi_p^{-1}(\mu_p\xi))\le 4 r_p<\delta_H,\label{mupxi}\ee 
for $p$ close enough to 1 ($\delta_H$ is defined in \eqref{deltah}).
With this in mind we can rewrite the asymptotic estimates \eqref{gij1}, \eqref{gij2}, \eqref{detg} for $\xi\in \Omega_p$ as follows:

 \be  |(g_p)_{ij}(\xi)-\delta_{ij}|= |g_{ij}^{(p)}(\mu_p\xi)-\delta_{ij}|\le3\rh^{-1}\mu_p|\xi| \le 12 \rh^{-1} r_p\label{gij12}\ee
 \be |g_p^{ij}(\xi)-\delta^{ij}|\le6n\rh^{-1} \mu_p|\xi|\le 24n\rh^{-1} r_p\label{gij22}\ee
\be\Big|\sqrt{|g_p|(\xi)}-1\Big|\le 3n^2\rh^{-1} \mu_p|\xi|\le 12 n^2 \rh^{-1}r_p, \label{detg2}\ee
where the first inequalities to the left of the above  are true if we only have $\mu_p|\xi|<\dh$, and also
\be\frac12|\xi-\eta|\le d_{\gp}(\xi,\eta)\le 2|\xi-\eta|\ee
\be B_{\sR^n}(\eta,{\tfrac12}r)\subseteq B_{\gp}(\eta,r)\subseteq B_{\sR^n}(\eta,2r),\qquad |\eta|<\mu_p^{-1}\dh,\; r\le \mu_p^{-1}\dh,\label{ballsp}\ee
in particular for {\emph {any}} $\eta,r$ if $p$ close enough to 1 (since $\mu_p\to0)$.
 
For later use, we also note the following estimate: for $p$ close enough to 1 and $\xi\in\Omega_p$
\be |\partial_k g_p^{ij}(\xi)|\le 12n^2 \rh^{-1} \mu_p\label{pgij}\ee
 
 To derive it, note that $\partial_k g_p^{ij}=-g_p^{i\ell}\partial_k (g_p)_{\ell h}g_p^{hj}$, as well as $g_p^{ij}\le 2$ for $p$ near 1, and 
 $|\partial_k (g_p)_{\ell h}(\xi)|=\mu_p|\partial_k g_{\ell h}^{(p)}(\xi)|\le 3\rh^{-1}\mu_p $, by definition of the harmonic norm.

 \medskip
If $\Delta_{p,\gp}$ denotes the $p-$Laplacian in the metric $\gp$, then we have the identities

\be \|v_p\|_{r,\gp}=\mu_p^{\frac n p-\frac n r-1}\|u_p\|_r,\label{v1}\ee

\be\|\nabla_\gp v_p\|_{r,\gp}=\mu_p^{\frac n p-\frac n r}\|\nabla u_p\|_r,\label{v2}\ee

\be- \Delta_{p,\gp} v_p+B\mu_p\|v_p\|_{p,\gp}^{1-p}\|\nabla_\gp v_p\|_{p,\gp}^{p-1}v_p^{p-1}=\lambda_p\|\nabla_\gp v_p\|_{p,\gp}^{p-1}v_p^{q-1}.\label{E2}\ee

\medskip
These can be brutally checked using the formulas 
\be |\nabla f|_g=(g^{k\ell}\partial_k f\partial_\ell f)^{\frac12}\label{grad}\ee
\be {\text {div}}_g X=\frac{1}{\sqrt{|g|}}\partial_i \big(\sqrt{|g|} X^i\big)\label{div}\ee
\be \Delta_{p,g}f =\frac 1{\sqrt{|g|}}\partial_i\Big(\sqrt{|g|}\big(g^{k\ell}\partial_k f\partial_\ell f\big)^{\frac{p-2}2}g^{ij}\partial_j f\Big)=g^{ij}
\partial_i\Big(\big(g^{k\ell}\partial_k f\partial_\ell f\big)^{\frac{p-2}2}\partial_j f\Big)\label{plap}
\ee
where the identity  in \eqref{plap} is valid in harmonic coordinates, since in such coordinates  $\partial_i\big(\sqrt{|g|} g^{ij})=0$.
In particular, note 
\be \|v_p\|_{q,\gp}=\|u_p\|_q=1\label{vpq}\ee
\be \|\nabla_\gp v_p\|_{p,\gp}=\|\nabla u_p\|_p.\label{pp}\ee

Now let
\be \vt_p=v_p^{\frac {p(n-1)}{n-p}}\label{vtilde}\ee
so that $\vt_p^\nnn=v_p^q$. We now claim that
\be \lim_{p\to1}\frac{\|\nabla \vt_p\|_1}{\|\vt_p\|_{\nnn}}=K(n,1)^{-1}.\label{sob1}\ee

Indeed from the sharp Sobolev inequality  on $\R^n$ we have $\|\nabla \vt_p\|_1\ge K(n,1)^{-1}{\|\vt_p\|_{\nnn}}$ which implies
\be \liminf_{p\to1}\frac {\|\nabla \vt_p\|_1}{\|\vt_p\|_{\nnn}}\ge K(n,1)^{-1}.\ee

On the other hand, from 
 \eqref{h2a}, \eqref{grad} (applied to $g_p$) we have
\be \frac14|\nabla \vt_p|^2\le |\nabla_\gp\vt_p|_\gp^2\le 4|\nabla \vt_p|^2,\ee
whereas \eqref{gij22} and \eqref{grad} imply that
\be\ba &\big||\nabla_\gp\vt_p|_\gp^2-|\nabla \vt_p|^2\big|\le6n\rh^{-1} \mu_p|\xi| \Big(\sum_{k=1}^n |\p_k \vt_p|\Big)^2\cr&\le 6n^2\rh^{-1}\mu_p|\xi| |\nabla \vt_p|^2\le  24n^2\rh^{-1} \mu_p|\xi| |\nabla_\gp \vt_p|_\gp^2\ea\ee
hence
\be  |\nabla \vt_p|^2\le|\nabla_\gp\vt_p|_\gp^2\big(1+ 24n^2\rh^{-1}\mu_p|\xi| \big)
\ee
and finally
\be |\nabla \vt_p|\le|\nabla_\gp\vt_p|_\gp\big(1+12 n^2\rh^{-1} \mu_p|\xi| \big).\ee
Using the above and \eqref{detg2}, \eqref{mupxi}, we get that, for $\xi\in \Omega_p$  (and $p$ close to 1)
\be{\frac12}\le \sqrt{|g_p|(\xi)}\le 2,\ee\
\be \Big| \sqrt{|g_p|(\xi)}-1\Big|
\le 3n^2\rh^{-1}\mu_p|\xi|\le 6n^2\rh^{-1}\mu_p|\xi|\sqrt{|g_p|(\xi)} \le 24n^2\rh^{-1} r_p\sqrt{|g_p|(\xi)},
\label{vest}\ee
and
\be\ba \|\nabla \vt_p\|_1&\le\int_{\R^n} |\nabla_\gp \vt_p|_\gp\big(1+ 12n^2 \rh^{-1}\mu_p|\xi| \big)\big(1+3n^2 \rh^{-1}\mu_p|\xi|\big)dV_\gp(\xi)\cr&
\le \int_{\R^n} |\nabla_\gp \vt_p|_\gp\big(1+16n^2\rh^{-1}\mu_p|\xi|\big)dV_\gp(\xi)
\le \|\nabla_\gp \vt_p\|_{1,\gp}\big(1+64n^2\rh^{-1} r_p\big)\ea\label{1}\ee
Arguing as in \eqref{eq1}
\be\ba \|\nabla_\gp \vt_p\|_{1,\gp}&=\frac{p(n-1)}{n-p}\int_{\R^n}|\nabla_\gp v_p|_\gp v_p^{\frac {n(p-1)}{n-p}}dV_\gp(\xi)\cr&\le \frac{p(n-1)}{n-p}\|\nabla_\gp v_p\|_{p,\gp}\|v_p\|_{q,\gp}=
 \frac{p(n-1)}{n-p}\|\nabla u_p\|_p\cr&= \frac{p(n-1)}{n-p}(\lambda_p-B\|u_p\|_p)\to K(n,1)^{-1}
\ea\label{2}
\ee
which together with 
\be\ba\|\vt_p\|_{\nnn}&=\int_{\R^n}v_p^q\;d\xi\ge\int_{\R^n}v_p^q\big(1-6n^2\rh^{-1}\mu_p|\xi|)\big)dV_\gp(\xi)\cr&\ge (1-24n^2\rh^{-1} r_p)\int_{\R^n}v_p^qdV_\gp(\xi)\to1 \ea\label{3} \ee
(where we used \eqref{vest}) yields $\limsup_{p\to1}{\|\nabla \vt_p\|_1}/{\|\vt_p\|_{\nnn}}\le K(n,1)^{-1}$, and hence \eqref{sob1}.

Now, \eqref{sob1} implies that $\{\vt_p\}$ is uniformly bounded in $W^{1,1}(\R^n)$ and therefore there is $v_0\in BV(\R^n)$ such that $\vt\to v_0$ weakly in $BV(\R^n)$, i.e. $\vt_p\to v_0$ in $L^1(K)$ for any $K$ compact, and $\int \phi\cdot \nabla \vt_p d\xi\to\int \phi\cdot Dv_0$, for any $\phi$ smooth, compactly supported and valued in $\R^n$.

Proceeding as in \cite{druet-pams} and \cite{noa-arxiv},  we get that   $\vt_p\to \chi_{B_{\R^n}^{}(\xi_0,R_0)}$ strongly in $BV(\R^n)$, (in particular in $L^\nnn(\R^n)$), for some $\xi_0\in\R^n$ and  where $|B_{\R^n}^{}(0,R_0)|=B_nR_0^n=1$. 

\begin{lemma} \label{strongestimate} For  each  $b\in (0,n-1)$, if  $p$ close enough to 1, we have
\be v_p(\xi)\le\min\bigg\{1, \bigg(\frac {8R_0}{|\xi|}\bigg)^{\frac{n-p-b}{p-1}}\bigg\},\qquad \xi\in\R^n.\label{decay} \ee
\end{lemma}

\ni{\bf Proof.} To establish this lemma we begin by deriving the following preliminary ``weak'' estimate: for $p$ close to 1 we have 
\be v_p(\xi) \le \min \bigg\{1, \bigg(\frac {8R_0}{|\xi|}\bigg)^{\frac n p-1}\bigg\},\qquad \xi\in\R^n\label{l1}.\ee

First, let's establish that for $R>2R_0$
\be \lim_{p\to1} \sup_{B_{g_p}(\xi_0,R)^c } v_p(\xi)=0\label{outside2}\ee
Suppose not, then assuming by contradiction that  there is $\z_p\in\Omega_p\setminus B_\gp(\xi_0,R)$ such that 
\be v_p(\z_p)=\max_{B_\gp(\xi_0,R)^c} v_p(\xi)\ge a>0\ee
(maybe along a sequence, with $p\to1$). Note that  from \eqref{ballsp} we get, for $p$ close to 1, 
\be|\z_p-\xi_0|\ge\tfrac12 R.\label {n1}\ee

For any fixed $\delta>0$ we can apply the Moser iteration to $v_p\le 1$ as defined on $B_{\gp}(\z_p,\delta/2)$
to get
\be 0<a\le v_p(\z_p)\le \sup_{B_{\gp}(\z_p,\delta/2)}v_p(\xi)\le C\int_{B_{\gp}(\z_p,\delta)} v_p^q d V_{\gp}.\label{xx}\ee
Given $|\xi-\z_p|<2\delta$ we have  (from \eqref{n1})
\be|\xi-\xi_0|\ge|\z_p-\xi_0|-|\xi-\z_p|\ge \tfrac12R-2\delta>R_0\ee
provided we pick $\delta>0$ so that $R>2R_0+4\delta$. Therefore,  
\be \int_{B_{\gp}(\z_p,\delta)} v_p^q d V_{\gp}\le 2\int_{B_{\sR^n}(\z_p,2\delta)} v_p^q d\xi\le 2\int_{B_{\sR^n}(\xi_0,R_0)^c} v_p^qd\xi,\label {n2}\ee
and since $\vt_p\to \chi_{B(\xi_0,R_0)}$ in $L^\nnn$, 
the right hand side of \eqref{n2} tends to 0 as $p\to1$, contradicting \eqref{xx}, and so \eqref{outside2} is true.
  
It's important to observe at this point that, since $v_p(0)=1$ for all $p$, then \eqref{outside2} implies that for all  $R>2R_0$, we have $0\in B_\gp(\xi_0,R)$, and hence $|\xi_0|< 2R$ for all $R>2R_0$, which implies 

 \be |\xi_0|\le 4R_0\label{xi0}.\ee

\medskip
To prove \eqref{l1}, let
\be w_p(\xi)=|\xi|^{\frac n p-1}v_p(\xi)\label{prelim}\ee
and assume by contradiction that if $\z_p$ is so that $w_p({\z_p})=\max w_p$, then 
$$w(\z_p)>(8R_0)^{\frac np-1}$$
 as $p\to1$ (again, along some subsequence).

Let 
\be\nu_p^{1-\frac n p}=v_p(\zeta_p)\le 1\ee
so that $w_p(\z_p)=(|\z_p|/\nu_p)^{\frac n p-1}>(8R_0)^{\frac np-1}$, and therefore 
\be|\zeta_p|>\nu_p8R_0\ge 8R_0.\label{zpest}\ee

Now let 
\be \phi_p(\xi)=\nu_p^{\frac n p-1}v_p(\nu_p\xi+\z_p)\label{phip}\ee
\be \wt \Omega_p=\nu_p^{-1}(\Omega_p-\z_p)\label{omegap}\ee
and let $\wt g_p$ be the metric on $\wt\Omega_p$  with components 
\be(\wt g_p)_{ij}(\xi)=(g_p)_{ij}(\nu_p\xi+\z_p)=g_{ij}^{(p)}(\mu_p(\nu_p\xi+\z_p)).\label{gtilde}\ee

Then, equations \eqref{v1}, \eqref{v2}, \eqref{E2} hold for $\phi_p$ instead of $v_p$, $\gpt$ instead of $\gp$, and  $\nu_p\mu_p$ instead of $\mu_p.$ In particular, \eqref{E2} implies
\be- \Delta_{p,\gpt} \phi_p+B\nu_p\mu_p\|\phi_p\|_{p,\gpt}^{1-p}\|\nabla_\gpt \phi_p\|_{p,\gpt}^{p-1}\phi_p^{p-1}=\lambda_p\|\nabla_\gpt \phi_p\|_{p,\gpt}^{p-1}\phi_p^{q-1}.\label{E2a}\ee

Note also that, as in \eqref{ballsp}, 
\be B_{\sR^n}(0,\tfrac12 r)\subseteq B_\gpt(0,r)\subseteq B_{\sR^n}(0,2r),\qquad r<(\nu_p\mu_p)^{-1}\dh\label{ballspt}\ee
and 
\be \sqrt{|\wt g_p|(\xi)}\le 2,\qquad \xi\in \wt\Omega_p.\ee

We now have that $\phi_p(\xi)\le 2^{\frac n p-1}$ for $\xi\in B_\gpt(0,R_0\nu_p^{-1})$, if $p$ near 1. Indeed, assuming $\xi\in B_\gpt(0,R_0\nu_p^{-1})$, and since $R_0<\mu_p^{-1}\dh$ for $p$ close to 1, from \eqref{ballspt},  we have $|\xi|<2R_0\nu_p^{-1}$, therefore (using \eqref{zpest})
\be|\nu_p\xi+\z_p|\ge|\z_p| -\nu_p|\xi| \ge|\z_p|-2R_0\ge \frac12|\z_p|.\label{a1}\ee
  Then, 
\be \phi_p(\xi)=\nu_p^{\frac n p-1}|\nu_p\xi+\z_p|^{-\frac n p  +1}w_p(\nu_p\xi+\z_p)\le 2^{\frac n p-1}\nu_p^{\frac n p -1}|\z_p|^{-\frac n p +1}w_p(\zeta_p)=2^{\frac n p -1}.\label{a2}
\ee
From \eqref{E2a}  we get  
\be- \Delta_{p,\gpt} \phi_p\le \lambda_p\|\nabla_\gpt \phi_p\|_{p,\gpt}^{p-1} \phi_p^{q-1}= \lambda_p\|\nabla u_p\|_{p}^{p-1} \phi_p^{q-1}\le \lambda_p^p\phi_p^{q-1}\le K(n,1)^{-2} \phi_p^{q-1},\ee
(using \eqref{eq2}, \eqref{L1}),  and the Moser iteration lemma (\cite{noa-arxiv}  Lemma 3.1)
gives 
\be \ba1&=\phi_p(0)\le  \sup_{B_{\wt g_p}(0,\frac12R_0\nu_p^{-1})} \phi_p\le C \bigg(\int_{B_{\gpt}(0,R_0\nu_p^{-1})} \phi_p^q \;dV_{\gpt}\bigg)^{\frac 1q}\le C
\bigg(2\int_{B_{\sR^n}(0,2R_0\nu_p^{-1})} \phi_p^q \;d\xi \bigg)^{\frac 1q}\cr&=C\bigg(\int_{B_{\sR^n}(0,2R_0\nu_p^{-1})} \nu_p^n v_p^q(\nu_p\xi+\z_p) \;d\xi\bigg)^{\frac 1q}= C \bigg( \int_{B_{\sR^n}(\z_p,2R_0 )} v_p^q(\xi) \;d\xi\bigg)^{\frac 1q}
\ea
 \label{contr}\ee
for $p$ close to 1.

On the other hand, we have $B_{\sR^n}(\z_p,2R_0)\subseteq B_{\sR^n}(\xi_0,R_0)^c$, since
 if $|\xi-\z_p|\le2R_0$ then 
 \be|\xi-\xi_0|\ge |\z_p-\xi_0|-|\xi-\z_p|\ge |\z_p|-|\xi_0|-2R_0\ge 2R_0\ee for $p$ close enough to 1 (we used \eqref{xi0}, \eqref{zpest}).
 Hence, from \eqref{detg2} and \eqref{ballspt} we obtain 
\be\int_{\Br(\z_p,2R_0)} v_p^q \;d\xi\le\int_{\Br(\xi_0,R_0)^c} v_p^q\;d\xi=\int_{\Br(\xi_0,R_0)^c} \vt_p^\nnn\;d\xi
\label{a3}\ee
and since $\vt_p\to \chi_{B(\xi_0,R_0)}$ in $L^\nnn$, from \eqref{a3} we get that 
\be\bigg(\int_{\Br(\z_p,2R_0 )} v_p^q \;d\xi\bigg)^{\frac 1q}\to0,\qquad p\to1\label{a4}\ee
thereby contradicting \eqref{contr}, which establishes \eqref{l1}

Now we obtain the following pointwise  estimate on $v_p$: for any $R>2R_0$ we have
\be \lim_{p\to1} \sup_{B_\gp(\xi_0,R)^c } |\xi|^{\frac n p-1} v_p(\xi)=0.\label{outside1}\ee

Suppose by contradiction that \eqref{outside1} fails, let
 
 \be w_p(\xi)=|\xi|^{\frac n p-1}v_p(\xi)\ee
  and let us find points $\z_p\in \Omega_p\setminus B_\gp(\xi_0,R)$ such that 
 \be w(\z_p)=\max_{B_{\gp}(\xi_0,R)^c} |\z_p|^{\frac n p-1}v_p(\z_p):=|\z_p|^{\frac n p-1}\nu_p^{-\frac n p+1}\ge a^{\frac np-1}>0.\ee
Because of \eqref{outside2} we have $v_p(\z_p)\to0$ hence $\nu_p\to+\infty$ and  $|\z_p|\ge a\nu_p\to+\infty$.

Defining $\phi_p, \wt\Omega_p, \wt\gp$ as in \eqref{phip}, \eqref{omegap}, \eqref{gtilde}, the proof of \eqref{outside1} follows as in the proof of \eqref{l1},  the only difference is that this time we have to use  $a/8$ rather than $R_0$.

\bigskip
Now consider the operator $L_p$ defined as 
\be L_p f=-\Delta_{p,\gp} f+{B} \mu_p \|v_p\|_{p,\gp}^{1-p}\|\nabla_\gp v_p\|_{p,\gp}^{p-1} f^{p-1}-\lambda_p\|\nabla_\gp v_p\|_{p,\gp}^{p-1} v_p^{q-p} f^{p-1}\ee
and let 
\be G_p(\xi)= |\xi|^{-\frac{n-p-b}{p-1}}.\ee
We want to prove now that for fixed $R>2R_0$ and $p$ near 1
\be L_p G_p(\xi)>0,\qquad \xi \in \Omega_p\setminus B_{\gp}(\xi_0,R).\label{pos}\ee
Once this estimate is in place then we can argue as follows. If  $\xi\in \overline{B_\gp(\xi_0,R)}$, then using \eqref{ballsp} and \eqref{xi0}  we get that $|\xi|\le |\xi-\xi_0|+|\xi_0|\le 2R+4R_0\le4 R$, so that for $p$ close enough to 1 
\be v_p(\xi)\le 1\le \bigg(\frac{4R}{|\xi|}\bigg)^{\frac {n-p-b}{p-1}}=\big(4R\big)^{\frac {n-p-b}{p-1}}G_p(\xi)\qquad \xi\in \overline{B_\gp(\xi_0,R)},\ee
in particular if $\xi\in \p B_{\gp}(\xi_0,R)$.
However since $L_pv_p=0$ on $\Omega_p$ using the comparison Lemma in \cite{aubin-li}  Lemma 3.4, we get that $v_p(\xi)\le\big(4R\big)^{\frac {n-p-b}{p-1}} G_p(\xi)$ in $\Omega_p\setminus B_\gp(\xi_0,R)$, and hence everywhere. Letting $R\to 2R_0$ gives \eqref{decay}.

\medskip
Now let us prove \eqref{pos}. Using \eqref{plap} we have
\be \p_j G_p=-\frac{n-p-b}{p-1} |\xi|^{-\frac{n-p-b}{p-1}-2} \xi_j\ee
\be |\nabla G_p|_\gp^{p-2}=\bigg(\frac{n-p-b}{p-1}\bigg)^{p-2} |\xi|^{-(p-2)\big(\frac{n-p-b}{p-1}+2\big)}\big(g_p^{k\ell}\xi_k \xi_\ell\big)^{\frac{p-2}2}\ee
\be g_p^{ij}\p_jG_p=-\frac{n-p-b}{p-1} |\xi|^{-\frac{n-p-b}{p-1}-2} g_p^{ij}\xi_j\ee

\be\ba&-\Delta_{p,\gp}G_p(\xi)=\bigg(\frac{n-p-b}{p-1}\bigg)^{p-1} \bigg[\Big(\partial_i |\xi|^{-n-p+b+2}\big(g_p^{k\ell}\xi_k \xi_\ell\big)^{\frac{p-2}2}g_p^{ij}\xi_j\cr&\qquad +|\xi|^{-n-p+b+2}\p_i\big(g_p^{k\ell}\xi_k \xi_\ell\big)^{\frac{p-2}2} g_p^{ij}\xi_j+
|\xi|^{-n-p+b+2}\big(g_p^{k\ell}\xi_k \xi_\ell\big)^{\frac{p-2}2}g_p^{ij}\p_i \xi_j\bigg]\cr&=
\bigg(\frac{n-p-b}{p-1}\bigg)^{p-1} \bigg[(-n-p+b+2)|\xi|^{-n-p+b}\big(g_p^{k\ell}\xi_k \xi_\ell\big)^{\frac{p}2}\cr&\qquad+|\xi|^{-n-p+b+2}\frac{p-2}2 \big(g_p^{k\ell}\xi_k \xi_\ell\big)^{\frac p2 -2}\Big(2g_p^{ik}\xi_k+(\partial_i g_p^{kl}\big)\xi_k\xi_\ell\Big)\big(g_p^{ij}\xi_j\big)\cr&\qquad+ |\xi|^{-n-p+b+2} \big(g_p^{k\ell}\xi_k \xi_\ell\big)^{\frac {p-2} 2}g_p^{ij}\delta_{ij}\bigg]\cr&
=\bigg(\frac{n-p-b}{p-1}\bigg)^{p-1} |\xi|^{-n+b}\bigg[(-n-p+b+2)\frac{\big(g_p^{k\ell}\xi_k\xi_\ell\big)^{\frac p2}}{|\xi|^p}+(p-2)\frac{\big(g_p^{k\ell}\xi_k\xi_\ell\big)^{\frac p2-2}\delta_{ki}g_p^{k\ell}\xi_\ell g_p^{ij}\xi_j }{|\xi|^{p-2}}\cr&\qquad +g_p^{ij}\delta_{ij}\frac{\big(g_p^{k\ell}\xi_k\xi_\ell\big)^{\frac p 2-1}}{|\xi|^{p-2}}+\frac{p-2}2 |\xi|\frac{\big(g_p^{k\ell}\xi_k\xi_\ell\big)^{\frac p2-2}}{|\xi|^{p-4}}\frac{\big(\p_ig_p^{k\ell}\big)\xi_k\xi_\ell g_p^{ij}\xi_j}{|\xi|^3}\bigg].
\ea\label{plap1}\ee

Using the  estimate in  \eqref{gij22} 
 we then have, for $p$ close enough to 1, 
 \be\big|g_p^{k\ell}\xi_k\xi_\ell-|\xi|^2\big|\le 24 n\rh^{-1} r_p \Big(\sum_k|\xi_k|\Big)^2\le 24n^2  \rh^{-1}r_p|\xi|^2\label{b1}\ee
  \be\big|\big(g_p^{k\ell}\xi_k\xi_\ell\big)^{\frac p 2}-|\xi|^p\big|\le 24n^2  \rh^{-1}r_p|\xi|^p\label{bb2}\ee
\be\big|\big(g_p^{k\ell}\xi_k\xi_\ell\big)^{\frac p 2-2}-|\xi|^{p-4}\big|\le 144n^2  \rh^{-1}r_p|\xi|^{p-4}\label{b3}\ee
 \be\big|\big(g_p^{k\ell}\xi_k\xi_\ell\big)^{\frac p 2-1}-|\xi|^{p-2}\big|\le 48n^2  \rh^{-1}r_p|\xi|^{p-2}\label{b31}\ee
\be\sum_i\big(g_p^{ik}\xi_k\big)^2=|\xi|^2+2 \xi_i (g_p^{ik}-\delta^{ik})\xi_k+\sum_i\Big(\big(g_p^{ik}-\delta^{ik}\big)\xi_k\Big)^2\label{b4}\ee
  \be\bigg|\sum_i\big(g_p^{ik}\xi_k\big)^2-|\xi|^2\bigg|\le 48n   \rh^{-1}r_p \Big(\sum_k|\xi_k|\Big)^2+24^2n^3 \rh^{-2}r_p^2|\xi|^2\le 100n^2  \rh^{-1}r_p|\xi|^2 \label{b5}\ee
  \be\bigg|\sum_i g_p^{ii}-n\bigg|\le24 n^2   \rh^{-1}r_p\label{b6}\ee
  (for \eqref{bb2} we used $(1+t)^a\le 1+a t$ and $(1-t)^a \ge 1-t$, when $0<a,t<1$, for \eqref{b3} we used $(1-t)^a\le 1+6t, (1+t)^a\ge 1-6t$, when $-3/2\le a\le -1$, and for \eqref{b31} we used $(1-t)^a\le 1+2t, (1+t)^a\ge 1-2t$, when $-1\le a< 0$).

  Also, using \eqref{pgij},
  \be|\p_i g_p^{k\ell}\xi_k\xi_\ell|\le 12n^2 \rh^{-1}\mu_p\bigg(\sum_k|\xi_k|\bigg)^2\le 12n^3\rh^{-1}\mu_p|\xi|^2\le 48n^3 \rh^{-1}r_p |\xi|\label{q1}\ee
  and
  \be\ba |\big(\p_ig_p^{k\ell}\big)\xi_k\xi_\ell g_p^{ij}\xi_j|&\le48n^3 \rh^{-1}r_p |\xi|^2 \sum_i |g_p^{ij}\xi_j|\le 48n^3 \rh^{-1}r_p |\xi|^2\sqrt 2\sqrt n|\xi|\le  48n^4 \rh^{-1}r_p|\xi|^3,
\ea\label{q2}  \ee
  where we used \eqref{b5}, which guarantees $\sum_i\big(g_p^{ij}\xi_j\big)^2\le 2|\xi|^2$ for $p$ close to 1.
  
  \medskip
  \begin{rk} We observe that the estimate on the partials of the metric tensor \eqref{gij0} is used only to derive the above estimates \eqref{q1}, \eqref{q2}.
  \end{rk}
  
  \medskip

  Putting these estimates together we find
  
  \be\ba -&\Delta_{p,\gp}G_p(\xi)\ge \bigg(\frac{n-p-b}{p-1}\bigg)^{p-1} |\xi|^{-n+b}\Big[(-n-p+b+2)(1-24n^2  \rh^{-1}r_p)\cr&+(p-2)(1-144n^2   \rh^{-1}r_p)(1-100n^2 \rh^{-1}r_p)+(n-24n^2 \rh^{-1}r_p)(1-48n^2  \rh^{-1}r_p)\cr&-\frac{p-2}2(1+144n^2  \rh^{-1}r_p)48n^4  \rh^{-1}r_p\;\Big]\ge \bigg(\frac{n-p-b}{p-1}\bigg)^{p-1} |\xi|^{-n+b}\big( b-C(n) \rh^{-1}r_p\big)\ea\ee
  where $C(n)$ is some constant depending only on $n$, and for $p$ close enough to 1.
  
  Finally,
  \be L_p G_p(\xi)\ge  |\xi|^{-n+b}\bigg[\bigg(\frac{n-p-b}{p-1}\bigg)^{p-1} \big( b-C(n) \rh^{-1}r_p\big)-\lambda_p\|\nabla_\gp v_p\|_{p,\gp}^{p-1}v_p^{q-p}(\xi)|\xi|^p\bigg]. \label{brackets}
  \ee
  Since  for  $\xi\in \Omega_p\setminus B_{\gp}(\xi_0,R)$ (with $R>2R_0$), 
 \be |\xi|^p v_p^{q-p}(\xi)=\Big(|\xi|^{\frac n p -1}v_p(\xi)\Big)^{\frac{p^2}{n-p}}\le\sup_{\Omega_p\setminus B_\gp(\xi_0,R)}\Big(|\xi|^{\frac n p -1}v_p(\xi)\Big)^{\frac{p^2}{n-p}}\to0
 \ee
as $p\to1$, from \eqref{outside1},   and using  \eqref{p2}, \eqref{p3}, \eqref{p4}, \eqref{pp}, \eqref{outside1}, we get that the expression inside brackets in \eqref{brackets} converges to a positive constant uniformly on $\Omega_p\setminus B_{\gp}(0,R)$, thereby proving \eqref{pos}, and therefore Lemma \ref{strongestimate}.

\qed
\smallskip
\begin{rk}\label{weakest}  Estimates \eqref{decay} and \eqref{l1} are a slight improvements of those given in \cite{druet-pams}, \cite{noa-arxiv}, in terms of the multiplicative constants on the right hand sides.

\bigskip
Now we move to the final part of the proof of Theorem \ref{druet}.  From the Sobolev inequality we have
\be\frac{\|\vt_p\|_\nnn}{\|\nabla\vt_p\|_1}\le K(n,1).\label{d8}\ee

The goal now is to improve the asymptotic estimates \eqref{1}, \eqref{3}, so as to have explicit dependence on $\mu_p$ rather than $r_p$.

From \eqref{vpq}, \eqref{vtilde} and \eqref{detg2} we have 
\be\ba\|\vt_p\|_{\nnn}&\ge\int_{\R^n}v_p^q\Big(\sqrt{|g_p|(\xi)}-3n^2\rh^{-1}\mu_p|\xi|\Big)d\xi =1-3n^2\rh^{-1} \mu_p\int_{\R^n} |\xi| v_p^q(\xi) d\xi.
\ea\ee
From the decay estimate \eqref{decay}, given any $R>0$, and with $\gamma_p=\frac{n-p-b}{p-1}$
\be\ba \int_{\R^n} |\xi| v_p^q(\xi) d\xi&\le \int_{|\xi|\le R } |\xi|d\xi+(8R_0)^{q\sgam_p }\int_{|\xi|\ge R} |\xi|^{1-q\sgam_p }d\xi\cr&=
\frac{\omega_{n-1}}{n+1}R^{n+1}\bigg[1+\bigg(\frac{8R_0}R\bigg)^{q\sgam_p} \frac{n+1}{q\gamma_p-n-1}\bigg].\label{d1}
\ea\ee
Since $0<b<n-1$ then for $p$ close enough to 1 we have $b<n-p$ and 
\be q\gamma_p-n-1=\frac{np}{n-p} \frac{n-p-b}{p-1}-n-1\sim \frac{n}{n-1}\frac{n-1-b}{p-1}-n-1\to+\infty,\ee
as $p\to1$, so choosing $R=8R_0$, and recalling that $|B_{\sR^n}(0,R_0)|=1$,  we get
\be\|\vt_p\|_\nnn\ge 1-\frac n{n+1}R_08^{n+1}3\rh^{-1}\mu_p\big(1+o(1)\big):=1-\mu_p\big(H_1+o(1)\big).\label{d7}
\ee

\bigskip
From \eqref{1}, \eqref{2} we have
\be\ba \|\nabla \vt_p\|_1
&\le \int_{\R^n} |\nabla_\gp \vt_p|_\gp\big(1+16n^2 \rh^{-1}\mu_p|\xi| \big)dV_\gp(\xi)\cr&\le \frac{p(n-1)}{n-p}(\lambda_p-{B}\|u_p\|_p)+16n^2\rh^{-1}\mu_p \int_{\R^n}  |\xi| |\nabla_\gp \vt_p|_\gp dV_\gp(\xi)\cr&=
\frac{p(n-1)}{n-p}(\lambda_p-{B}\mu_p\|v_p\|_{p,\gp})+16n^2\rh^{-1}\mu_p \int_{\R^n}  |\xi| |\nabla_\gp \vt_p|_\gp dV_\gp(\xi).
\ea\label{E29}
\ee

The next goal is to show that the integral in \eqref{E29} is bounded, using the decay estimate \eqref{decay}.

First, we have

\be\ba &\int_{\R^n}  |\xi| |\nabla_\gp \vt_p|_\gp dV_\gp=\frac{p(n-1)}{n-p}\int_{\R^n}  |\xi|v_p^{\frac{n(p-1)}{n-p}}  |\nabla v_p|_\gp dV_\gp\cr&\le \frac{p(n-1)}{n-p}\bigg(\int_{\R^n} v_p^{p'\frac{n(p-1)}{n-p}} dV_{\gp}\bigg)^{\frac1{p'}}\bigg(\int_{\R^n}  |\xi|^p |\nabla_{\gp} v_p|_{\gp}^p dV_\gp\bigg)^{\frac1p} \cr&=\frac{p(n-1)}{n-p}\bigg(\int_{\R^n}  |\xi|^p |\nabla_{\gp} v_p|_{\gp}^p dV_\gp\bigg)^{\frac1p}.\label{d0}
\ea\ee

Multiplying \eqref{E2} by $|\xi|^p v_p$ and integrating by parts we obtain

\be\ba &\int_{\R^n} |\nabla_\gp v_p|_\gp^{p-2}\langle\nabla_\gp (|\xi|^p v_p),\nabla_\gp v_p\rangle_\gp dV_\gp\cr
&=\lambda_p\|\nabla_\gp v_p\|_{p,\gp}^{p-1}\int_{\R^n} |\xi|^pv_p^q dV_\gp-B\mu_p\|v_p\|_{p,\gp}^{1-p}\|\nabla_\gp v_p\|_{p,\gp}^{p-1}\int_{\R^n} |\xi|^p v_p^pdV_\gp
\cr&\le 
\lambda_p\|\nabla_\gp v_p\|_{p,\gp}^{p-1}\int_{\R^n} |\xi|^pv_p^q dV_\gp\le 2\lambda_p\|\nabla_\gp v_p\|_{p,\gp}^{p-1}\int_{\R^n} |\xi|^pv_p^q d\xi.
\ea\ee
Using the decay bound \eqref{decay} as in \eqref{d1}
\be\ba&\int_{\R^n}|\xi|^p v_p^q(\xi)d\xi\le\int_{|\xi|\le 8R_0}|\xi|^pd\xi+(8R_0)^{q\sgam_p} \int_{|\xi|\ge 8R_0}|\xi|^{p-q\sgam_ p}d\xi\cr&
=\frac{\omega_{n-1}}{n+p}(8R_0)^{n+p}\bigg(1+\frac1{q\gamma_p-n-p}\bigg)=\frac n{n+1} 8^{n+1} R_0\big(1+o(1)\big),\label{d2}
\ea\ee
and this gives (using \eqref{L1}, \eqref{p3})
\be\ba&\int_{\R^n} |\nabla_\gp v_p|_\gp^{p-2}\langle\nabla_\gp (|\xi|^p v_p),\nabla_\gp v_p\rangle_\gp dV_\gp
\le 2\lambda_p^p\frac n{n+1} 8^{n+1} R_0\big(1+o(1)\big)\cr&
<2\Big(K(n,1)^{-1}\frac{n-p}{p(n-1)}\Big)^p \frac n{n+1} 8^{n+1} R_0\big(1+o(1)\big)=\frac {2n^2}{n+1} 8^{n+1}\big(1+o(1)\big).\label{d3}
\ea\ee
On the other hand (where as usual $\xi^*=\xi/|\xi|$)
\be\ba&\int_{\R^n} |\nabla_\gp v_p|_\gp^{p-2}\langle\nabla_\gp (|\xi|^p v_p),\nabla_\gp v_p\rangle_\gp dV_\gp=\int_{\R^n}
| \nabla_\gp v_p|_\gp^p |\xi|^pdV_\gp\cr&+p\int_{\R^n} v_p|\xi|^{p-1}  |\nabla_\gp v_p|_\gp^{p-2}\langle \xi^*, \nabla_\gp v_p\rangle_\gp dV_\gp \ge \int_{\R^n}
| \nabla_\gp v_p|_\gp^p |\xi|^pdV_\gp\cr&-p\int_{\R^n} v_p|\xi|^{p-1}  |\nabla_\gp v_p|_\gp^{p-1} dV_\gp.\label{d4}
 \ea\ee

From H\"older's and Young's inequalities  
\be \ba p\int_{\R^n} v_p|\xi|^{p-1}  |\nabla_\gp v_p|_\gp^{p-1} dV_\gp&\le p \|v_p\|_{p,\gp}\bigg(\int_{\R^n} |\xi|^{p}  |\nabla_\gp v_p|_\gp^{p} dV_\gp\bigg)^{\frac {p-1}p}\cr&
\le \frac{p^p \|v_p\|_{p,\gp}^p} p+\frac{1}{p'}\int_{\R^n} |\xi|^{p}  |\nabla_\gp v_p|_\gp^{p} dV_\gp
\label{d5}\ea\ee
then, combining  \eqref{d3}, \eqref{d4}, \eqref{d5}, we get
\be \int_{\R^n} |\xi|^{p}  |\nabla_\gp v_p|_\gp^{p} dV_\gp\le \frac {2n^2p}{n+1} 8^{n+1}\big(1+o(1)\big)+p^p \|v_p\|_{p,\gp}^p.\label{d5a}
\ee
Using that $v_p\to \chi_{B(\xi_0,R_0)}$ a.e. (along a subsequence)  and the decay estimate \eqref{decay},  we have
\be \|v_p\|_{p,\gp}^p=(1+O(r_p))\int_{\R^n}v_p^p d\xi=1+o(1)\ee
putting this togeter with  \eqref{d0}, \eqref{d5a} we get
\be \int_{\R^n}  |\xi| |\nabla_\gp \vt_p|_\gp dV_\gp\le 1+\frac {2n^2}{n+1} 8^{n+1}+o(1)\le 2n 8^{n+1}
\ee
and from \eqref{E29}, recalling that $K(n,1)=R_0/n$, 
\be\ba \|\nabla \vt_p\|_1
&\le
K(n,1)^{-1}-\frac{p(n-1)}{n-p}B\mu_p\|v_p\|_{p,\gp}+16n^2\rh^{-1}\mu_p 2n 8^{n+1}+o(\mu_p)\cr&
<K(n,1)^{-1} \bigg(1+\mu_p\Big(-\frac{R_0}n B+32n^2 8^{n+1} R_0 r_H^{-1}\Big)\bigg)+o(\mu_p) \cr&:=
K(n,1)^{-1}\Big(1+\mu_p\big(H_2+o(1)\big)\Big).
\ea\label {d6}\ee

Combining \eqref{d8},  \eqref{d7} and \eqref{d6} we get
\be K(n,1)\ge\frac{\|\vt_p\|_\nnn}{\|\nabla\vt_p\|_1}>K(n,1)\frac{ 1-\mu_p\big(H_1+o(1)\big)}{1+\mu_p\big(H_2+o(1)\big)}=K(n,1)\Big(1-\mu_p\big(H_1+H_2+o(1)\big)\Big),
\ee
and,  finally, the above bound implies $H_1+H_2\ge0$, i.e.
\be B\le \Big(32n^3+\frac{3n^2}{n+1}\Big) 8^{n+1} r_H^{-1}.\label{B}\ee
This shows that if $B$ larger than the right-hand side of \eqref{B},  then $r_0>0$ can be found so that \eqref{sobdruet} holds for all $u\in W_0^{1,1}(B(x,r_0))$, and all $x\in M$.

\qed
\vskip1em
\section{A first order isoperimetric profile expansion}\label{6}

In this section we prove the following:

\begin{theorem}\label{expansion} Suppose that on $M$ we have $|\ric|\le K$ and $\inj(M)>0$. Then, there exist $\tau_0,B>0$ such that  
\be \im(v)\ge nB_n^{\frac1n} v^\nn- B v\qquad 0<v\le \tau_0.\label{cond1}\ee

\end{theorem}
Note that \eqref{cond1} clearly implies
\be I_0(v)\ge 1-(nB_n)^{-\frac1n} B v^{\frac1n},\qquad 0<v\le \tau_0.\ee

\bigskip
\ni{\bf Proof.} Under the given assumptions there exists $v_0>0$ such that for $0<v<v_0$ at least one of the following two cases occurs:

\medskip
\ni(A) There exists $\Omega\subseteq M$ with finite perimeter such that 
\be \vol(\Omega)=v,\qquad \im(v)=\per(\Omega)\label{isregion}\ee

\smallskip
\ni(B) There exists a pointed $C^{1,\alpha}$ Riemannian manifold $(M_0, x_0, g_0)$ and a sequence $\{x_j\}$ on $M$ such that $(M,x_j,g)\to (M_0,x_0,g_0)$ in the $C^{1,\a}$ topology. Moreover, there exists  $\Omega_0\subseteq M_0$ with finite perimeter such that 
\be \vol_{g_0}(\Omega_0)=v,\qquad \im(v)=\per_{g_0}(\Omega_0).\label{isregion0}\ee

\smallskip
Recall that $(M,x_j,g)\to (M_0,x_0,g_0)$ in the $C^{1,\a}$ topology means the following: for any compact set $K\subseteq M_0$, with $x_0\in K$ there exist, up to a subsequence, compact sets $K_j\subseteq M$, $x_j\in K_j$, and $C^{2,\alpha}$ diffeomorphisms $\Phi_j:K\to K_j$ such that $\Phi_j^{-1}(x_j)\to x_0$ and $\Phi_j^* g$   converges in $C^{1,\alpha}$  to $g_0$, in the induced $C^{2,\alpha}$ complete atlas of $K$.

\smallskip

In essence, for small enough volumes, one can guarantee the existence of isoperimetric regions in the original manifold or ``at infinity''. This concept has been first studied by Nardulli in \cite{nardulli-asian} and \cite{nardulli-calcvar}, where it is proved, roughly speaking,  that given any minimizing sequence of domains for $\im(v)$ (of  volume $v$), then up to a subsequence each domain splits into finitely many pieces each of which converges  to a domain  on a pointed limit manifold, or the manifold itself, and the union of such limit domains is an isoperimetric region of  volume $v$. Additionally, if the volume $v$ is small enough then there is only one such limit region. This result is stated in \cite{nardulli-asian}, Theorem 2, which is valid under our assumption \eqref{ric} (note that in that theorem, property (X) is stated to be  valid in $C^m$ geometry with $m\ge4$ but it is still valid under $C^{1,\alpha}$ bounded geometry, e.g. under \eqref{ric}, see third remark on page 6). For additional information, see  \cite{apps-annales} Lemma 4.18, in the more general context of RCD spaces, building on results obtained in \cite{anp} Theorem 1.1, Corollary 1.3.

 Returning to the proof, one of the features of an isoperimetric region is that its diameter is well controlled by its volume: If $\ric\ge K$ and $\inj(M)>0$ then there exist $C_0>0$, and $v_0>0$ such that any isoperimetric region $\Omega$ satisfies  
\be \diam(\Omega)\le C_0 \vol(\Omega)^{\frac1n},\qquad {\text{ whenever }}\;\vol(\Omega)\le v_0.\label{diam-vol}\ee

This result was proved first in \cite{noa-imrn} Lemma 4.9 under bounded sectional curvature and $\inj(M)>0$, whereas a more general version is obtained in \cite{apps-annales} Prop. 4.21, valid in the general context of RCD spaces, which includes Riemannian manifolds satisfying $\ric\ge K$ and $\inj(M)>0$, and their pointed limits. Let us  then choose 
\be \tau_0=\min\bigg\{\bigg(\frac {r_0}{2C_0}\bigg)^n,v_0\bigg\}\label{tau1}\ee
where $r_0$ is as in Theorem \ref{druet}. 
With this choice, if  case (A) occurs with  $\vol(\Omega)=v\le\tau_0$, then $\diam(\Omega)\le r_0/2< r_0$, and \eqref{isop1} of Corollary \ref{druet1} gives \eqref{cond1} directly.

If case (B) occurs with $\vol_{g_0}(\Omega_0)=v\le\tau_0$, we still have $\diam(\Omega_0)\le r_0/2$, but in order to apply Corollary \ref{druet1} we need to approximate $\Omega_0$ with suitable domains  living  on $M$. To that effect, we  can find a  sequence of domains $\{\Omega_j\}$ in $M$ with finite perimeter, and such that as $j\to\infty$, 
\be \diam(\Omega_j)  \to \diam(\Omega_0),\qquad \vol(\Omega_j)\to \vol_{g_0}(\Omega_0),\qquad \per(\Omega_j)\to \per_{g_0}(\Omega_0).\label{seq}\ee

This can be achieved  via the diffeomorphisms $\Phi_j$ defining the $C^{1,\alpha}$ pointed  convergence $(M,x_j,g)\to (M_0,x_0,g_0)$, associated with a compact set $K=\overline {B(x_0,R)}$, where $R$ is large enough so as to contain $\Omega_0$ (which is bounded by \eqref{diam-vol}): simply let $\Omega_j=\Phi_j(\Omega_0)$. The properties in \eqref{seq} are then guaranteed by  the $C^{1,\alpha}$ convergence of the metrics $\Phi_j^* g$ (cf. \cite{nardulli-calcvar} proof of  Lemma 3.3).

Since $\diam(\Omega_j)<r_0$ for large $j$,  we can now apply Corollary  \ref{druet1} to get
\be \per(\Omega_j)\ge nB_n^{\frac1n} \vol(\Omega_j)^{\frac{n-1}n}-B\;\vol(\Omega_j),
\ee
and passing to the limit as $j\to\infty$ we obtain the result. 

\qed

\medskip
\begin{rk} We note that the sequence $\{\Omega_j\}$ defined in the above proof is not necessarily a minimizing sequence for $I_M$. The domain $\Omega_0$ is in fact constructed from a minimizing sequence $\{\wt\Omega_j\}$, in such a way that (up to a subsequence, and for $v$ small enough) $\vol(\wt\Omega_j)=v\to\vol(\Omega_0)$, and $\per(\wt\Omega_j)\to I_M(v)= \per(\Omega_0)$. However, we are not guaranteed that the diameter of the domains  $\wt\Omega_j$ are small enough in order to apply Corollary \ref{druet1} to those sets. 

\vskip1em
\section{Proof of Theorem \ref{thm}}\label{7}

The proof is an immediate consequence of Theorems \ref{loc-glob}, \ref{smallprofile-sobolev}, \ref{expansion}.

\qed 

\bigskip
\section{Appendix}\label{appendix}

Here is an outline of the proof of Proposition \ref{polyaszego}, assuming $u\in C_c^\infty(\Omega)$, $u\ge0$ (for general $u$ apply the argument to $|u|$, which is Lipschitz and smooth on $\{u\ne0\}$, the only region seen by the level sets $\{|u|=t\}$, $t>0$). 

If  $\phi\in L^1(M)$ then (coarea formula) 
\be \int_M\phi\;|\nabla u|d\mu=\int_0^\infty dt \int_{ u^{-1}(t)} \phi\; d{\mathcal H}_{n-1}\label {coarea}\ee
where ${\mathcal H}_{n-1}$ denotes the $(n-1)$-dim. Hausdorff measure (formula is also valid if $u$ is only Lipschitz). 

Let $L=\|\nabla u\|_\infty$ and 
\be \mu(t)=\mu(\{u>t\}).\ee
For $0<t'<t<\max u$, apply \eqref{coarea} with $\phi=\chi_{\{t'<u\le t\}}$, use $\mathcal H^{n-1}(\{u=s\})=\mathrm{Per}(\{u>s\})$ for a.e. $s$ to  get 


\be  L\big(\mu(t')-\mu(t)\big) \ge\int_{t'<u\le t}|\nabla u|d\mu =\ \int_{t'}^{t}\mathrm{Per}(\{u>s\})\,ds \label{stima1}\ee

Now we have the following:

\noindent{\bf Fact 1.}{\emph{ If $E\subseteq M$ is measurable,  $\vol(E)<\vol(M)$, and $\per(E)=0$, then $\vol(E)=~0$.}

\smallskip
The proof of this is based on the fact that if $\per(E)=0$ then $\nabla \chi_E=0$, and since $M$ connected $\chi_E$ can either be $0$ or $1$ a.e. on $M$.

\smallskip
\noindent{\bf Fact 2.}  {\emph{ If $K$ is compact in $M$ and $\vol(K)<\vol(M)$, then there is $C_K>0$ such that 
\be  \per(E)\ge C_K\vol(E)^{\frac{n-1}n},\qquad {\text {for all measurable  }}\; E\subseteq K.\label{relisop}\ee 
}}
\smallskip
The proof of \eqref{relisop} for $\mu(E)\le v_0$, some $v_0>0$ small enough, is an application of a uniform relative isoperimetric inequality for small balls in \cite{ritore} Lemma 1.46, followed by a covering of $K$ by finitely many balls. To settle the case $\mu(E)\ge v_0$, suppose that there exist $E_j\subseteq K$ with $\mu(E_j)\ge v_0$ and $\per(E_j)\to0$. Then by $BV$ compactness, modulo a subsequence $\chi_{E_j}\to \chi_E$ in $L^1$ for some  $E\subseteq K$ measurable. In particular $\mu(E)\ge v_0$ and $\per(E)=0$ (by lower semicontinuity of the perimeter). Since $M$ is connected and $\vol(K)< \vol(M)$, Fact 1 implies $\mu(E)=0$ which is a contradiction. Hence there is $c_1>0$ such that $\mu(E)\ge v_0$ implies $\per(E)\ge c_1$, which gives \eqref{relisop}. 

Back to \eqref{stima1}, we can apply Fact 2 to $K=\supp u$ (note that $\vol(K)<\vol(M)$ since $M$ is noncompact, so $M\setminus K$ is nonempty and open) and  $E=\{u>s\}$ to deduce
\be C_K\mu(t)^{\frac{n-1}n}(t-t')\le L\big(\mu(t')-\mu(t)\big)\label{stima2}\ee
Inequality \eqref{stima2} is  (23a) of \cite{talenti} p. 363,  with $nB_n^{\frac1n}$ replaced by $C_K$. The Lipschitz bound  $|u^\#(\xi)-u^\#(\eta)|\le\frac{nB_n^{\frac1n}L}{C_K}\big||\xi|-|\eta|\big|$ follows by applying \eqref{stima2} on a fine partition of $[t',t]$ (this step is stated without proof in \cite{talenti} p. 363). So $u^\#$ is Lipschitz on $\mathbb R^n$ and   it vanishes outside $\Omega^\#$, so for every $p<\infty$ we have $u^\#\in W_0^{1,p}(\Omega^\#)$. Indeed, $(u^\#-\varepsilon)_+$ is Lipschitz with compact support in $\Omega^\#$, and it converges to $u^\#$ in $W^{1,p}$ as $\varepsilon\to0$.

\smallskip

By Sard's theorem,  a.e. $t\in(0,\max u)$ is a regular value of $u$, that is, the level set $u^{-1}(t)=\{u=t\}$ is a smooth $(n-1)$-dimensional submanifold of $M$,  $\nabla u\neq0$ on that level set,  $\mu(\{u=t\})=0,$  $\mu(t)=\mu(\{u>t\})$ is continuous at $t$, and strictly decreasing, by \eqref{stima2}, therefore $u^*(\mu(t))=t$.  
Also, for $t$ regular we have, from \eqref{coarea}
\be \mu'(t)=-\int_{u=t} \frac {d\mathcal H_{n-1}}{|\nabla u|}\ee
and for $t$  outside a null set, so that $\mu$ is differentiable at $t$ and $u^*$ is differentiable
at $\mu(t)$
\be \mu'(t)=-\int_{u^\#=t}\frac {d\mathcal H_{n-1}}{|\nabla u^\#|}=-\frac{{\mathcal H}_{n-1}(\{u^\#=t\})}{\big|\nabla u^\#_{/\{{u^\#=t\}}}\big|}\label{tt2}\ee
The last identities come  from the fact that   if $t$ is a regular value of $u$ with $\{u^\#>t\}=B(0,|\xi|)$, then  $\mu(t)=B_n|\xi|^n$, $u^*(\mu(t))=t$,   $(u^*)'(\mu(t))\mu'(t)=1$, and $|\nabla u^\#(\xi)|=(-u^*)'(B_n|\xi|^n)nB_n|\xi|^{n-1}$.

Using H\"older's inequality and arguing as in \cite{talenti} pp. 361-362, we can obtain, for a.e. $t$,
\be\int_{u=t} |\nabla u|^{p-1} d{\mathcal H}_{n-1}\ge (-\mu'(t))^{1-p}\;{\mathcal H}_{n-1}(\{u=t\})^p= (-\mu'(t))^{1-p}\;\per(\{u>t\})^p\label {tt1}\ee
and from the definition of $I_0(v)$
\be \per(\{u>t\})\ge \mu(t)^{\frac{n-1}n}nB_n^{\frac1n} I_0(\mu(t))=\per(\{u^\#>t\})I_0(\mu(t)).\ee

Let $\Omega^\#=B_{\sR^n}(0,R)$ and $\tilde u(r)=u^*(B_nr^n)$, a decreasing  Lipschitz function. Then $\big|\nabla u^\#_{/\{{u^\#=t\}}}\big|=-\tilde u'(r(t))$, where $r(t)$ is the unique $r$ s.t. $\tilde u(r)=t$ when $t$ is regular (i.e. a.e). Using the coarea formula \eqref{coarea}, and \eqref{tt2}
\be\ba\int_\Omega|\nabla u|^p d\mu&=\int_M|\nabla u|^{p-1}|\nabla u| d\mu=\int_0^{\max u} dt \int_{u=t} |\nabla u|^{p-1} d{\mathcal H}_{n-1}\cr&
\ge\int_0^{\max u}(-\mu'(t))^{1-p}\per(\{u^\#>t\})^p I_0(\mu(t))^pdt\cr&=
\int_0^{\max u} (-\tilde u'(r(t)))^{p-1} \per(\{u^\#>t\}) I_0(\mu(t))^pdt
\ea\ee
On the other hand,  by polar coordinates
 \be\ba&\int_{\Omega^\#}|\nabla u^\#|^p I_0(B_n|\xi|^n)^p d\xi=\int_0^R (-\tilde u'(r))^{p} n B_n r^{n-1}I_0(B_nr^n)^pdr\cr&=\int_0^R (-\tilde u'(r))^{p-1} n B_n r^{n-1}I_0(B_n r^n)^p (-\tilde u'(r))dr\cr&=
\int_0^{\max u} (-\tilde u'(r(t)))^{p-1} \per(\{u^\#>t\}) I_0(\mu(t))^pdt\le\int_\Omega|\nabla u|^p d\mu, \ea\ee 
where for the last  identity we used the 1-dimensional coarea formula applied to the Lipschitz function $-\tilde u$.

\qed

\ni{\bf Data availability statement.} Data sharing is not applicable, since no data were used for this research.

\medskip
\ni {\bf Conflict of interest.} The authors declare no conflicts of interest, whether direct or indirect, related to this work.

\vskip1em
\begin{table}[h]

\setlength{\tabcolsep}{24pt} 
\begin{tabular}{@{}p{0.5\linewidth}p{0.5\linewidth}@{}}
\textbf{Carlo Morpurgo} & \textbf{Stefano Nardulli} \\
Department of Mathematics & Centro de Matem\'atica\\
University of Missouri & Cogni\c c\~ao  e Computa\c c\~ao\\
Columbia, Missouri 65211 & Universidade Federal do ABC \\
USA & Santo Andr\'e, SP \\
\texttt{morpurgoc@umsystem.edu} & Brazil \\
{}&\texttt{stefano.nardulli@ufabc.edu.br} \\\\
\textbf{Liuyu Qin} & \\
Department of Mathematics and Statistics & \\
Hunan University of Finance and Economics & \\
Changsha, Hunan & \\
China & \\
\texttt{Liuyu\_Qin@outlook.com} & \\
\end{tabular}
\end{table}
\end{document}